\documentclass[11pt]{amsart}
\usepackage{geometry}                
\usepackage{graphicx}
\usepackage{epstopdf}
\usepackage{amsmath}
\usepackage{amssymb}

\usepackage{amsthm}
\usepackage{thmtools}
\declaretheorem[name=Theorem,numberwithin=section]{thm}
\declaretheorem[name=Lemma,sibling=thm]{lemma}

\declaretheorem[name=Proposition,sibling=thm]{prop}
\declaretheorem[name=Corollary,sibling=thm]{cor}

\declaretheorem[name=Example,style=remark]{exmp}
\usepackage{nameref,hyperref,cleveref}
\allowdisplaybreaks
\usepackage{lipsum}
\usepackage{enumitem}
\setlist[enumerate]{itemsep=0mm}

\title[Limit Theorems Regarding Products of Random Matrices I]{Some Limit Theorems Regarding Products of Random Matrices I: Directional Derivative of the Lyapunov Exponent}
\author{Fan Wang}

\begin{document}
\maketitle
\begin{abstract}
Given an i.i.d. sequence $\{A_n(\omega)\}_{n\ge1}$ of invertible matrices and a random matrix $B(\omega)$, we consider the random matrix sequences inductively defined by $S_{n}(\omega) = A_n(\omega)S_{n-1}(\omega)$ and $T_{n}(\omega) = B(\sigma^{n-1}\omega)S_{n-1}(\omega) + A_n(\omega)T_{n-1}(\omega)$, and study several limit theorems involving $T_{n}(\omega)$ as well as the asymptotic behaviour of the action of $T_{n}(\omega)$ on the projective space and on the unit circle.
\end{abstract}

\section{Introduction}
\label{sec:intro}

\subsection{Background}
Let $(\Omega,\mathcal{A},\mathbb{P},\sigma)$ be a probability space on which we can define an i.i.d. sequence $\{A_n(\omega)\}_{n\ge 1}$ taking values in $d\times d$ real invertible matrices following a probability measure $\mu$. Here we assume the measure-preserving map $\sigma$ is invertible. In particular, one can take $\Omega = GL(d,\mathbb{R})^{\mathbb{Z}}$, let $A(\omega)=\omega_0$ and $A_n(\omega) = A(\sigma^{n-1} \omega)$, where $\sigma:\Omega \to\Omega$ is the shift operator satisfying $(\sigma\omega)_i = \omega_{i+1}$ for $i \in\mathbb{Z}$.

Define 
\[
S_n(\omega) := A_n(\omega)A_{n-1}(\omega)\cdots A_1(\omega).
\]
It has been of great interest studying the following limits 
\begin{align}
&\lim_{n\to\infty}\dfrac{1}{n}\mathbb{E}[\log\|S_n(\omega)\|],\label{lim:le}\\
&\lim_{n\to\infty}\dfrac{1}{n}\log\|S_n(\omega)\|, \label{lim:norm} \\
&\lim_{n\to\infty}\dfrac{1}{n}\log\|S_n(\omega)x\|, \quad (0\not = x\in\mathbb{R}^d) \label{lim:x}
\end{align}
where $\|\cdot\|$ can be chosen to be any matrix (or vector) norm. 

For example, limit \eqref{lim:le} exists under the finite expectation condition
\[
\mathbb{E}[\sup\{\log\|A_1(\omega)\|, 0\}] < \infty.
\]
The existence of the limit follows easily from the sub-additive property, namely $\log\|S_{m+n}\|\le \log\|S_m\| + \log\|S_n\|$. Limit \eqref{lim:le} is also the definition of the (top) Lyapunov exponent, which characterises the average growth rate of the system under logarithm.

The relationships between the other limits \eqref{lim:norm} \eqref{lim:x} and the Lyapunov exponent were fully explored. Furstenberg and Kesten \cite{Furstenberg1} proved that limit \eqref{lim:norm} almost surely exists and equals the Lyapunov exponent under the same finite expectation condition. In \cite{Furstenberg3}, Furstenberg and Kifer studied the action of $S_{n}(\omega)$ on the projective space $\mathcal{P}(\mathbb{R}^d)$ and proved that under mild conditions, \eqref{lim:x} also almost surely gives the Lyapunov exponent. 
\subsection{Main Results}
In this paper, we plan to study another type of limit theorem which can be considered as an analogue of the first directional derivative of the Lyapunov exponent. 

Given a probability space $(\Omega,\mathcal{A},\mathbb{P}, \sigma)$ on which we can define an i.i.d. sequence $\{(A_n(\omega)\}_{n\ge 1}$ of $d\times d$ invertible matrices as well as a random matrix $B(\omega)$ (not necessarily invertible) with induced i.i.d. sequence $B_n(\omega) := B(\sigma^{n-1}\omega)\ (n\ge1)$, define
\begin{align}\label{eq:t.formula}
T_n(\omega) := &B_n(\omega)A_{n-1}(\omega)\cdots A_1(\omega)  \\ &+ A_n(\omega)B_{n-1}(\omega)A_{n-2}(\omega)\cdots A_1(\omega) + \dots
 + A_n(\omega)\cdots A_2(\omega)B_1(\omega). \nonumber
\end{align}
Alternatively, we can define $T_n(\omega)$ inductively by 
\begin{equation}
\label{defn:t}
T_n(\omega) = B_n(\omega) S_{n-1}(\omega) + A_n(\omega) T_{n-1}(\omega),
\end{equation}
where $S_n(\omega) = A_n(\omega)S_{n-1}(\omega)$ as above. For convenience, we call $T_{n}(\omega)$ the \emph{first derived product} associated to $(A_1(\omega), B(\omega))$, denoted by $T_{n}(\omega):=\mathfrak{D}_n(\omega, A_1, B)$.

Here we plan to study the following limits
\begin{align}
&\lim_{n\to\infty}\dfrac{1}{n}\dfrac{\|T_n(\omega)\|_2}{\|S_n(\omega)\|_2}, \label{lim:t/s}\\
&\lim_{n\to\infty}\dfrac{1}{n}\dfrac{\|T_n(\omega)x\|_2}{\|S_n(\omega)x\|_2}, \label{lim:t/sx}\\
&\lim_{n\to\infty}\dfrac{1}{n}\dfrac{\langle S_n(\omega)x, T_n(\omega)x\rangle}{\|S_n(\omega)x\|_2^2}, \label{lim:in.prod}
\end{align}
where $x$ is a non-zero deterministic vector in $\mathbb{R}^d$ and $\|\cdot\|_2$ denotes the $\ell_2$-norm of matrices (or vectors), as well as the asymptotic behaviour of $T_{n}(\omega)$ acting on the projective space $\mathcal{P}(\mathbb{R}^d)$ and on the unit circle $\mathbb{S}^{d-1}$.

Under certain conditions, by defining a random variable
\begin{equation}\label{intro.defn:xi}
\xi(\omega)=\dfrac{\langle \widetilde{Z}(\omega), B(\omega)Z(\omega) \rangle}{\langle \widetilde{Z}(\omega), A(\omega)Z(\omega)\rangle},
\end{equation}
where $\widetilde{Z}(\omega)$ and $Z(\omega)$ are two random vectors (formally defined in \Cref{sec:z}) which can be considered as the limit points of $S_{n}^*(\omega)x/\|S_{n}^*(\omega)x\|$ and $S_{n}(\sigma^{-n}\omega)x/\|S_{n}(\sigma^{-n}\omega)x\|$ respectively, we prove that \eqref{lim:t/s} and \eqref{lim:t/sx} both converge to $|\mathbb{E}[\xi(\omega)]|$ almost surely (see \autoref{cor:t.asymp}\ref{t:ts} and \autoref{thm:varphi.x}), and \eqref{lim:in.prod} converges to $\mathbb{E}[\xi(\omega)]$ almost surely (see \autoref{thm:psi}). We also have an integral formula for the limit given by \eqref{eq:xi.int}, that is
\[
\mathbb{E}[\xi(\omega)]= \iiint \dfrac{\langle [y], B[x] \rangle}{\langle [y], A[x] \rangle}\operatorname{d}\mu(A, B)\operatorname{d}\nu([x])\operatorname{d}\nu^*([y]),
\]
where $\nu$ is a $\mu$-invariant measure on the projective space, $\mu^*$ is the law of the i.i.d. sequence $\{A_n^*(\omega)\}_{n\ge1}$ of transposes, and $\nu^*$ is a $\mu^*$-invariant measure on the projective space (the invariant measures exist under our conditions).

Moreover, for any deterministic vector $x\not = 0$, we prove that $S_{n}(\omega)x$ and $T_{n}(\omega)x$ align to the same direction (up to a sign) almost surely as $n\to\infty$ if $\mathbb{E}[\xi(\omega)]\not = 0$ (see \autoref{cor:alignment.x}), and the sign of the direction is determined by the sign of $\mathbb{E}[\xi(\omega)]$ (see \autoref{cor:xy.conv}). Also any limit point of $T_{n}^*(\omega)/\|T_{n}^*(\omega)\|$ is almost surely a rank one matrix (see \autoref{cor:t.asymp} \ref{t:range}).

\subsection{Key Ideas}
If one assumes that any limit point of $S_{n}^*(\omega)/\|S_{n}^*(\omega)\|$ is almost surely a rank one matrix with range being the 1-dimensional vector space $W(\omega)$ (that is, satisfying the contracting property), then each term 
\begin{equation}\label{eq:term}
A_1^*(\omega)\cdots A_{i-1}^*(\omega)B_i^*(\omega)A_{i+1}^*(\omega)\cdots A_n^*(\omega)
\end{equation}
of $T_{n}^*(\omega)$ almost surely converges to a rank one matrix after normalisation as $n\to\infty$. However, sum of rank one matrices is most likely not a rank one matrix unless they all share the same range, and $B_i^*(\omega)$ in the term \eqref{eq:term} constantly prevents them from having the same range. This suggests the problem is not trivial.

Notice when $i$ in the term \eqref{eq:term} is sufficiently large, the term has a range very close to the range $W(\omega)$ of $S_{n}^*(\omega)/\|S_{n}^*(\omega)\|$. It is natural to think that we can separate $T_{n}^*(\omega)$ into two parts, one being the sum of terms with $i$ sufficiently large (denoted by $P_\infty$), and the other with the rest of the terms (denoted by $P_{0}$). If the norm of $P_{\infty}$ is way larger than the norm of $P_{0}$, then we can expect that the range of $T_{n}^*(\omega)/\|T_{n}^*(\omega)\|$ is dominated by the range of $P_{\infty}/\|P_\infty\|$. However, there is a problem in this approach. Note \eqref{eq:term} is a product of $n$ matrices, thus when $n\to\infty$, the norm of $P_0$ might explode (when the Lyapunov exponent associated to $S_{n}(\omega)$ is positive) even though $P_0$ is a finite sum. Meanwhile, even though $P_{\infty}$ is an infinite sum of matrices with almost the same range, its norm is not necessarily big, since $B(\omega)$ can give each term a different sign so that their norms cancel with each other (see \Cref{sec:exmp} \autoref{exmp2}). 

We get around with this problem with the help of $\mathbb{E}[\xi(\omega)]$. Inspired by Ruelle's work \cite{Ruelle79}, we first prove in \autoref{prop:varphi.z} that under mild conditions
\begin{equation}\label{eq:intro.varphi.z}
\lim_{n\to\infty}\dfrac{1}{n}\dfrac{\|T_{n}(\omega)Z(\omega)\|_2}{\|S_{n}(\omega)Z(\omega)\|_2} = |\mathbb{E}[\xi(\omega)]|\quad \text{ a.s.}.
\end{equation}
This implies that when $\mathbb{E}[\xi(\omega)]\not = 0$, $\|T_{n}(\omega)Z(\omega)\|_2/\|S_{n}(\omega)Z(\omega)\|_2\to\infty$ almost surely. This gives the dominance of $P_\infty$ in norms as we wanted above, except that it is only acting on $Z(\omega)$. That is, we obtain in \autoref{cor:tz} that $S_{n}(\omega)Z(\omega)$ and $T_{n}(\omega)Z(\omega)$ align to the same direction as $n\to\infty$. Furthermore, based on this result, for a given deterministic vector $x\not = 0$, by comparing $\|T_{n}(\omega)x\|_2/\|S_{n}(\omega)x\|_2$ and $\|T_{n}(\omega)Z(\omega)\|_2/\|S_{n}(\omega)Z(\omega)\|_2$, we prove in \autoref{thm:varphi.x} that
\begin{equation}\label{eq:intro.varphi.x}
\lim_{n\to\infty}\dfrac{1}{n}\dfrac{\|T_{n}(\omega)x\|_2}{\|S_{n}(\omega)x\|_2} = |\mathbb{E}[\xi(\omega)]|\quad \text{ a.s.}.
\end{equation}
With the same arguments (see \autoref{thm:alignment}), we have $S_{n}(\omega)x$ and $T_{n}(\omega)x$ align to the same direction as $n\to\infty$. Finally we use these results to compute limit \eqref{lim:in.prod}.

\subsection{Outline} 
In \Cref{sec:motivations}, we explain multiple motivations for the definition of $T_{n}(\omega)$ including products of random matrices in the reducible case, as well as the first directional derivative of the Lyapunov exponent. \Cref{sec:fes} introduces several finite expectation conditions which will be used throughout the paper. \Cref{sec:z} gives the definitions of two random directions $\widetilde{Z}(\omega)$ and $Z(\omega)$ as well as their properties. These two random directions will be used to define $\xi(\omega)$ in \Cref{sec:varphi.z}, where we will also give the proof of \eqref{eq:intro.varphi.z}. \Cref{sec:ind} gives the inductive formula of $\|T_{n}(\omega)x\|_2/\|S_{n}(\omega)x\|_2$ and studies the visits of $T_{n}(\omega)x$ at $0$. \Cref{sec:angles} gives a formal proof of the key idea above, proving $S_{n}(\omega)x(\omega)$ and $T_{n}(\omega)x(\omega)$ align to the same direction if $\|T_{n}(\omega)x(\omega)\|_2/\|S_{n}(\omega)x(\omega)\|_2$ converges to $\infty$ almost surely. \Cref{sec:varphi.x} studies limits \eqref{lim:t/s} \eqref{lim:t/sx} and the limit points of $T_{n}^*(\omega)/\|T_{n}^*(\omega)\|$, while limit \eqref{lim:in.prod} is studied in \Cref{sec:psi.x}. Results in \Cref{sec:varphi.x} and \Cref{sec:psi.x} are under the assumption $\mathbb{E}[\xi(\omega)]\not = 0$. In \Cref{sec:drop}, we generalise some of the results by dropping this assumption using a simple linear trick. \Cref{sec:exmp} gives several examples relevant to the problem, including a trivial example of commutative scalar products, an example of $T_{n}(\omega)$ visiting $0$ infinitely often almost surely, and examples of failures of the theory without the conditions assumed in our results.

\subsection{Notations and Basic Definitions.} 
For functions $f$ and $g$, denote $f^+:=\sup\{f, 0\}$ and $f^-:=\sup\{-f,0\}$. In particular, $\log^+f:=\sup\{\log f, 0\}$. We denote by $\|\cdot\|$ a general norm of matrices or vectors, and use $\|\cdot\|_2$ to specify the $\ell_2$-norm of matrices or vectors. 

For a vector $x\in\mathbb{R}^d$, if $x\not = 0$, then we denote by $[x]$ its class in the projective space $\mathcal{P}(\mathbb{R}^d)$, representing the direction of the vector $x$, and denote by $\bar{x}$ the normalised vector $x/\|x\|$ in the unit circle $\mathbb{S}^{d-1}$. Note that $\bar{x}$ and $\overline{-x}$ share the same direction $[x]$. Furthermore, if $A\in M(d,\mathbb{R})$ is a $d\times d$ matrix, we can define an action of $A$ on $\mathcal{P}(\mathbb{R}^d)$ by $A\cdot [x] = [Ax]$ if $\|Ax\| \not = 0$. If $\mu$ is a probability measure on matrices, we say $\nu$ is a $\mu$-invariant measure if it is a probability measure on $\mathcal{P}(\mathbb{R}^d)$ invariant under the action defined above.
Given a vector subspace $V$ of $\mathbb{R}^d$, denote the orthogonal complement of $V$ by
\[
V^{\perp}:= \{u\in\mathbb{R}^d: \langle u, v \rangle = 0, \forall \ v\in V\}.
\]
For a vector $x\not = 0$, denote by $x^{\perp}$ the orthogonal complement of the vector subspace generated by vector $x$. Given a matrix $M$, we use $M^*$ to denote the transpose of that matrix.

When two random variables $\xi(\omega)$ and $\eta(\omega)$ follow the same law, we write $\xi(\omega)\sim_d\eta(\omega)$. Given a subset $S\subseteq GL(d,\mathbb{R})$, we say $S$ is \emph{irreducible} if there is no proper linear subspace $V$ of $\mathbb{R}^d$ which is invariant under each element of $S$. We say $S$ is \emph{strongly irreducible} if there do not exist a finite family of linear subspaces $V_1, \dots, V_k$ such that $V_1\cup\dots\cup V_k$ is invariant under each element of $S$. We say $S$ is \emph{contracting} if there is a sequence $\{M_n\}_{n\ge1}$ of matrices in $S$ such that $M_n/\|M_n\|$ converges to a rank one matrix.

When $\mu$ is a probability measure supported on $GL(d,\mathbb{R})$, let $\mathcal{S}_\mu$ be the closed semigroup of $GL(d,\mathbb{R})$ generated by the matrices in the support of $\mu$. We say $\mu$ is irreducible(resp. strongly irreducible, contracting) if $\mathcal{S}_\mu$ is irreducible (resp. strongly irreducible, contracting). Finite expectation conditions (FE1)-(FE4) are introduced at the end of \Cref{sec:fes}.

For a given random vector $x(\omega)\not = 0$, for simplicity, we will write 
\begin{equation}\label{defn:x}
\begin{aligned}
x_m(\omega, x) := \dfrac{S_{m}(\omega)x}{\|S_{m}(\omega)x\|}, \quad
y_m(\omega, x) := \dfrac{T_{m}(\omega)x}{\|T_{m}(\omega)x\|}, \quad
\varphi_m(\omega, x):= \dfrac{\|T_{m}(\omega)x\|}{\|S_{m}(\omega)x\|}.
\end{aligned}
\end{equation}
When $\|\cdot\|$ is chosen to be $\ell_2$-norm $\|\cdot\|_2$, we also define
\begin{equation}\label{defn:psi}
\psi_m(\omega, x) := \dfrac{\langle S_{m}(\omega)x, T_{m}(\omega)x \rangle}{\|S_{m}(\omega)x\|_2^2}.
\end{equation}
In the article, $x_m, y_m, \varphi_m, \psi_m$ will be used for convenience when there is no ambiguity. Note $y_m$ is well-defined only when $T_{m}(\omega)x \not = 0$, while since each matrix $A_n(\omega)$ is assumed to be invertible, $x_m, \varphi_m$ and $\psi_m$ are always well-defined. A discussion of when $y_m$ is well-defined can be found in \Cref{sec:ind}.

\section{Motivations}
\label{sec:motivations}

The construction of $T_n(\omega)$ above seems artificial at the first glance. In this section, we explain several basic ways to understand $T_n(\omega)$ and how it is relevant to a few seemingly different questions which have not been fully investigated in the area of products of random matrices.
\subsection{Products of Random Matrices: Reducible Case}
If $\{(A_n(\omega), B_n(\omega))\}_{n\ge 1}$ is an i.i.d. sequence of $d\times d$ real matrix pairs, then we can define $2d\times 2d$ matrices
\[
C_n(\omega) := \begin{pmatrix}
A_n(\omega) & B_n(\omega) \\
O & A_n(\omega)
\end{pmatrix}.
\]
Now $\{C_n(\omega)\}_{n\ge 1}$ is an i.i.d. sequence following a reducible probability measure, as $\mathbb{R}^d\times 0^d\subset \mathbb{R}^{2d}$ is a common invariant subspace for all matrices. Define $\widetilde{S}_n(\omega) = C_n(\omega)\cdots C_1(\omega)$, then we can immediately see that
\[
\widetilde{S}_n(\omega) = \begin{pmatrix}
S_n(\omega) & T_n(\omega)\\
O & S_n(\omega)
\end{pmatrix}
\]
by the inductive formula \eqref{defn:t}. It is shown in \cite{Furstenberg3} \cite{Hennion84} that, if $\gamma$ is the top Lyapunov exponent associated to $S_{n}(\omega)$, then it is also the top Lyapunov exponent associated to $\widetilde{S}_{n}(\omega)$. Moreover, for $\epsilon >0$ we have $\|T_{n}(\omega)\|\le e^{n(\gamma+\epsilon)}$ when $n$ is sufficiently large. However, this estimation is quite crude. For example, when $B_n(\omega) \equiv 0$, we have $|T_{n}(\omega)|\equiv 0$. \Cref{sec:exmp} \autoref{exmp2} gives an example in which $B_n(\omega)\not\equiv 0$ but $T_{n}(\omega) = 0 \ \text{i.o.}$. In this paper we give more refined estimates such as (see \autoref{cor:t.asymp} \ref{t:ts})
\[
\lim_{n\to\infty}\dfrac{1}{n}\dfrac{\|T_{n}(\omega)\|_2}{\|S_{n}(\omega)\|_2} = |\mathbb{E}[\xi(\omega)]|
\]
under mild conditions, where $\xi(\omega)$ is defined in \eqref{intro.defn:xi} and \eqref{defn:xi}, with $\mathbb{E}[\xi(\omega)]=0$ in previous two examples.

The study of the random matrix products in reducible case is also often related to the study of random difference equations. For example, \cite{Kesten73} studied equations of the form
\[
T_{n}(\omega) = A_n(\omega)T_{n-1}(\omega) + Q_n(\omega),
\]
where the pair $\{(A_n(\omega),Q_n(\omega))\}_{n\ge1}$ is i.i.d.. By comparing with our definition \eqref{defn:t}, we can see that our problem replaces $Q_n(\omega)$ by $B_n(\omega)S_{n-1}(\omega)$, which is no longer independent. In fact, Kesten's problem can be considered as the random matrix products of the form
\[
\begin{pmatrix}
A_n(\omega) & Q_n(\omega) \\
O & I_d
\end{pmatrix}.
\]

We remark here, in our problem since 
\[
\widetilde{S}_{n+m}(\omega) = \widetilde{S}_n(\sigma^m\omega)\widetilde{S}_m(\omega) = 
\begin{pmatrix}
S_{n}(\sigma^m\omega) & T_{n}(\sigma^m\omega) \\
O & S_{n}(\sigma^m\omega)
\end{pmatrix}
\begin{pmatrix}
S_{m}(\omega) & T_{m}(\omega) \\
O & S_{m}(\omega)
\end{pmatrix},
\]
comparing the $(1,2)$-block on both sides gives another inductive formula
\begin{equation}
\label{eq:t.ind}
T_{m+n}(\omega) = S_{n}(\sigma^m\omega)T_{m}(\omega) + T_{n}(\sigma^m\omega)S_{n}(\omega).
\end{equation}
This will be frequently used in later part of the paper.

\subsection{Products of Random Matrices over $\mathbb{R[\epsilon]}/(\epsilon^2)$}
Given the same i.i.d. sequence of real matrix pairs $\{(A_n(\omega), B_n(\omega))\}_{n\ge 1}$, then $\{A_n(\omega) + \epsilon B_n(\omega)\}_{n\ge 1}$ is an i.i.d. sequence of matrices over the ring $\mathbb{R}[\epsilon]/(\epsilon^2)$. One can easily verify
\[
(A_n(\omega) + \epsilon B_n(\omega)) \cdots (A_1(\omega) + \epsilon B_1(\omega)) = S_n(\omega) + \epsilon T_n(\omega).
\]
Therefore the study of random matrix products over $\mathbb{R[\epsilon]}/(\epsilon^2)$ reduces to the study of $S_{n}(\omega)$ and $T_{n}(\omega)$.

\subsection{First Directional Derivative of the Lyapunov Exponent}
Let $\epsilon>0$ and regard $A_{n}(\omega) + \epsilon B_{n}(\omega)$ as a perturbation of the matrix $A_n(\omega)$ in the direction $B$. We still assume $\{(A_n(\omega), B_n(\omega))\}_{n\ge 1}$ is an i.i.d. sequence of real matrix pairs with $A_n(\omega)$'s all invertible. Let $\gamma(\epsilon)$ be the Lyapunov exponent associated to the perturbed problem, that is
\[
\gamma(\epsilon) = \lim_{n\to\infty}\dfrac{1}{n}\log\|S_n^\epsilon(\omega)\|,
\]
where $S_{n}^\epsilon(\omega):=(A_{n}(\omega) + \epsilon B_{n}(\omega))\cdots (A_{1}(\omega) + \epsilon B_{1}(\omega)) = S_{n}(\omega) + \epsilon T_{n}(\omega) + O(\epsilon^2)$ as $\epsilon\to0$.

Now we choose the vector norm to be the $\ell_2-$norm, then
\begin{align*}
\log\|S_{n}^\epsilon(\omega)x\|_2^2 & = \log\|S_{n}(\omega)x+\epsilon T_{n}(\omega)x + O(\epsilon^2)\|_2^2\\
& = \log\left(\|S_n(\omega)x\|_2^2 + 2\epsilon\langle S_{n}(\omega)x, T_{n}(\omega)x\rangle + O(\epsilon^2) \right)\\
& = \log\|S_{n}(\omega)x\|_2^2 + \log\left(1 + 2\epsilon\dfrac{\langle S_{n}(\omega)x,T_{n}(\omega)x \rangle}{\|S_{n}(\omega)x\|_2^2} + O(\epsilon^2)\right)\\
& = \log\|S_{n}(\omega)x\|_2^2 + 2\epsilon\dfrac{\langle S_{n}(\omega)x,T_{n}(\omega)x \rangle}{\|S_{n}(\omega)x\|_2^2} + O(\epsilon^2).
\end{align*}
Therefore, letting $n\to\infty$ in the equation
\[
\dfrac{1}{n}\log\|S_{n}^\epsilon(\omega)x\|_2 = \dfrac{1}{n}\log\|S_{n}(\omega)x\|_2 + \dfrac{1}{n}\dfrac{\langle S_{n}(\omega)x, T_{n}(\omega)x \rangle}{\|S_{n}(\omega)x\|_2^2}\epsilon + O(\epsilon^2)
\]
gives
\begin{equation}\label{eq:derle}
\lim_{n\to\infty} \dfrac{1}{n} \dfrac{\langle S_{n}(\omega)x, T_{n}(\omega)x\rangle}{\|S_{n}(\omega)x\|_2^2} + O(\epsilon) = \dfrac{\gamma(\epsilon) - \gamma(0)}{\epsilon}.
\end{equation}
As we can see, the limit \eqref{lim:in.prod} appears on the left-hand side of \eqref{eq:derle}. Moreover, in \autoref{cor:alignment.x} we will prove that $S_{n}(\omega)x$ and $T_{n}(\omega)x$ align to the same direction almost surely as $n\to\infty$ under certain conditions. In this case, this limit can also be reduced to \eqref{lim:t/sx} immediately (up to a sign). Of course, we do not know if the limit on the left-hand side of \eqref{eq:derle} exists a priori, or if $\gamma$ is differentiable at $\epsilon = 0$, or if the terms in $O(\epsilon)$ might explode when $n\to\infty$. The derivation above is informal, but it motivates the study of the limits \eqref{lim:t/sx} and \eqref{lim:in.prod}.

The analyticity of the Lyapunov exponent was first studied by Ruelle \cite{Ruelle79} under the condition that the matrices preserve a family of cones. However, given a set of matrices, it is not clear how to verify if they preserve certain family of cones, nor how to construct the family of cones, except when the set consists of positive matrices (in which case they all preserve the positive orthant). In this paper, we do not intend to study the analyticity of the Lyapunov exponent, but the limits which can potentially give the first directional derivative of the Lyapunov exponent. Moreover, we study under the conditions of strong irreducibility and the contracting property, which are also usually easy to verify.

\section{Finite Expectation Conditions}
\label{sec:fes}

In this section, we introduce several finite expectation conditions and show how they give a crude bound for $\|T_{n}(\omega)\|/\|S_{n}(\omega)\|$.
\begin{lemma}\label{lm:st.bnd}
Assume $\mathbb{E}[\log^+\|A_1(\omega)\|] < \infty$. Then there exists an integer $N_1>0$ such that for any $n>N_1$, there exists $C_1(n)>0$ satisfying
\begin{align*}
\|S_{n}(\sigma^m\omega)\| \le C_1(n)\quad \text{ a.s.},
\end{align*}
for all $m\ge 0, \omega\in\Omega$. 
Moreover,
\begin{enumerate}[label=(\alph*)]
\item If one further has  $\mathbb{E}[\|B_1(\omega)A_1^{-1}(\omega)\|]<\infty$, then there exists an integer $N_2>0$ such that for any $n>N_2$, there exists $C_2(n)>0$ satisfying
\[
\|T_{n}(\sigma^m\omega)\|  \le C_2(n)\quad \text{ a.s.},
\]
for all $m\ge 0, \omega\in\Omega$. 
\item If one further has $\mathbb{E}[\log^+\|A_1^{-1}(\omega)\|] < \infty$, then there exists an integer $N_3>0$ such that for any $n>N_3$, there exists $C_0(n)>0$ satisfying
\[
C_0(n)\le \|S_{n}(\sigma^m\omega)x\| \le C_1(n)\quad \text{ a.s.},
\]
for all $m\ge 0, \omega\in\Omega$ and $x\in\mathbb{R}^d$ with $\|x\|=1$.
\end{enumerate}

\end{lemma}
\begin{proof}
First note
\[
\|S_{n}(\sigma^m\omega)\| \le \|A_n(\sigma^m\omega)\|\cdots\|A_1(\sigma^m\omega)\|.
\]
Therefore
\[
\log\|S_{n}(\sigma^m\omega)\| \le \sum_{i=1}^n \log\|A_i(\sigma^m\omega)\| \le \sum_{i=1}^n \log^+\|A_i(\sigma^m\omega)\|.
\]
This shows
\[
\dfrac{1}{n}\log\|S_{n}(\sigma^m\omega)\| \le \dfrac{1}{n}\sum_{i=1}^n \log^+\|A_i(\sigma^m\omega)\| \to \mathbb{E}[\log^+\|A_1(\omega)\|] \ (n\to\infty) \quad \text{ a.s.},
\]
by the law of large numbers. Fix an $\epsilon > 0$, then there exists an integer $N_1>0$ such that when $n > N_1$, we have
\[
\|S_{n}(\sigma^m\omega)\| \le \exp(n\mathbb{E}[\log^+\|A_1(\omega)\|] + n\epsilon) =: C_1(n)\quad \text{ a.s.}.
\]
Similarly, by \eqref{eq:t.formula} we have
\[
\|T_{n}(\sigma^m\omega)\| \le \|A_n(\sigma^m\omega)\|\cdots\|A_1(\sigma^m\omega)\|\sum_{i = 1}^n \|B_i(\sigma^m\omega)A_i^{-1}(\sigma^m\omega)\|.
\]
As we have seen above, the product part can be bounded by $C_1(n)$. As for the summation part, we have
\[
\dfrac{1}{n}\sum_{i=1}^n \|B_i(\sigma^m\omega)A_i^{-1}(\sigma^m\omega)\| \to \mathbb{E}\left[\|B_1(\omega)A_1^{-1}(\omega)\|\right] \ (n\to\infty) \quad \text{ a.s.},
\]
by the law of large numbers. Fix an $\epsilon > 0$, then there exists an integer $N_2>0$ such that when $n > N_2$, we have
\[
\|T_{n}(\sigma^m\omega)\| \le C_1(n) n \left(\mathbb{E}[\|B_1(\omega)A_1^{-1}(\omega)\|] + \epsilon\right) =: C_2(n)\quad \text{ a.s.}.
\]
Finally, for a matrix $M\in GL(d,\mathbb{R})$ and a normalised vector $x$, we have 
\[
0 = \log\|x\| \le \log\|M^{-1}\| + \log\|Mx\|,
\]
thus
\[
-\log\|Mx\|\le \log\|M^{-1}\| \le \log^+\|M^{-1}\|,
\]
comparing the positive part of both sides gives
\[
\log^-\|Mx\| \le \log^+\|M^{-1}\|.
\]
Together with $\log^+\|Mx\|\le\log^+\|M\|$, we have
\[
|\log\|Mx\|| = \log^+\|Mx\| + \log^-\|Mx\| \le \log^+\|M\| + \log^+\|M^{-1}\|.
\]
Consequently, applying $M = S_{n}(\sigma^m\omega)$ gives
\begin{align*}
\left|\dfrac{1}{n}\log\|S_{n}(\sigma^m\omega)x\|\right| & \le \dfrac{1}{n}\sum_{i=1}^n \left[\log^+\|A_i(\sigma^m\omega)\| + \log^+\|A_i^{-1}(\sigma^m\omega)\|\right] \\
& \to \mathbb{E}[\log^+\|A_1(\omega)\|] + \mathbb{E}[\log^+\|A_1^{-1}(\omega)\|] \ (n\to\infty) \quad \text{ a.s.},
\end{align*}
by the law of large numbers. For a fixed $\epsilon > 0$, when $n$ is sufficiently large, we have
\[
\|S_{n}(\sigma^m\omega)x\| \ge \exp(-n\mathbb{E}[\log^+\|A_1(\omega)\|] - n\mathbb{E}[\log^+\|A_1^{-1}(\omega)\| + n\epsilon] =: C_0(n)>0,
\]
where $C_0(n)>0$ follows from the finite expectation conditions.
\end{proof}
Note we cannot use only $\mathbb{E}[\log^+\|A_1(\omega)\|] < \infty$ and the top Lyapunov exponent $\gamma$ to obtain a lower bound for $\|S_{n}(\sigma^{m}\omega)x\|$ because $\gamma$ might equal to $-\infty$.

From now on, we will for convenience refer the finite expectation conditions by
\begin{enumerate}[label = (FE{\arabic*})]
\item $\mathbb{E}[\log^+\|A_1(\omega)\|] < \infty$;
\item $\mathbb{E}[\|B_1(\omega)A_1^{-1}(\omega)\|] < \infty$;
\item $\mathbb{E}[\log^+\|A_1^{-1}(\omega)\|] < \infty$.
\end{enumerate}
Moreover, by writing  $\ell(M) := \sup\{\log^+\|M\|, \log^+\|M^{-1}\|\}$ for a matrix $M$, we introduce
\begin{enumerate}[resume, label = (FE{\arabic*})]
\item for some $\tau>0$, $\mathbb{E}[\exp(\tau\ell(A_1(\omega)))]<\infty$.
\end{enumerate}

\section{Two Random Directions $Z(\omega)$ and $\widetilde{Z}(\omega)$}
\label{sec:z}
In this section we introduce two important random directions $Z(\omega)$ and $\widetilde{Z}(\omega)$ in $\mathcal{P}(\mathbb{R}^d)$ as well as their properties.

Assume the probability measure $\mu$ is strongly irreducible and contracting, and let $\nu$ be a $\mu$-invariant measure on $\mathcal{P}(\mathbb{R}^d)$. It is a classical result \cite[Theorem III.3.1]{Bougerol} that $S_n(\sigma^{-n}\omega)\nu$ converges weakly to a Dirac measure $\delta_{Z(\omega)}$, in the sense that, for any Borel function $f$ on $\mathcal{P}(\mathbb{R}^d)$,
\[
\int f(S_{n}(\sigma^{-n}\omega)\cdot [x])\operatorname{d}\nu([x]) \to f(Z(\omega)) \quad (n\to\infty).
\]
Here any limit point of $S_{n}(\sigma^{-n}\omega)/\|S_{n}(\sigma^{-n}\omega)\|$ is almost surely a rank one matrix with the range being the 1-dimensional vector space spanned by $Z(\omega)$. If we further have the finite expectation (FE4), then for a vector $x\not = 0$, $S_{n}(\sigma^{-n}\omega)\cdot [x]$ converges in direction to $Z(\omega)$ almost surely \cite[Theorem VI.3.1]{Bougerol}.

We can do the same for $S_{n}^*(\omega) = A_1^*(\omega)\cdots A_n^*(\omega)$. Consider $\{A_n^*(\omega)\}_{n\ge1}$ as an i.i.d. sequence following a probability measure, denoted by $\mu^*$. Note $\mu$ is strongly irreducible and contracting if and only if $\mu^*$ is strongly irreducible and contracting, since both of the conditions are preserved by matrix transpose. Let $\nu^*$ be a $\mu^*$-invariant measure, then $S_{n}^*(\omega)\nu^*$ converges weakly to a Dirac measure $\delta_{\widetilde{Z}(\omega)}$. Under (FE4), we also have $S_{n}^*(\omega)\cdot [x]$ converges to $\widetilde{Z}(\omega)$ almost surely.

The following lemma gives another description of $\widetilde{Z}(\omega)$, and is a corollary of \cite[Proposition III.3.2]{Bougerol}. Note here we take a normalised vector $\widetilde{z}(\omega)\in\mathbb{S}^{d-1}$ as  a representative of $\widetilde{Z}(\omega)\in\mathcal{P}(\mathbb{R}^d)$, and define $|\langle x,\widetilde{Z}(\omega) \rangle|:=|\langle  x,\widetilde{z}(\omega) \rangle|$. It is well-defined because both of the choices $\widetilde{z}(\omega)$ and $-\widetilde{z}(\omega)$ give the same value.

\begin{lemma}\label{lm:ss}
Let $\mu$ be strongly irreducible and contracting, then for any random vector $x(\omega)\not = 0$,
\begin{align*}
\lim_{n\to\infty}\dfrac{\|S_{n}(\omega)x(\omega)\|_2}{\|S_{n}(\omega)\|_2} =  |\langle x(\omega), \widetilde{Z}(\omega) \rangle| \quad \text{ a.s.},
\end{align*}
where $\widetilde{Z}(\omega)$ is the random direction spanning the range of any limit point of $S_{n}^*(\omega)/\|S_{n}^*(\omega)\|$. The convergence is uniform in $x(\omega)\not = 0$.
\end{lemma}
\begin{proof}
Let $S_{n}(\omega) = U_n(\omega) D_n(\omega) V_n(\omega)$ be the polar decomposition of the matrix $S_{n}(\omega)$, where $D_n(\omega) = \operatorname{diag}(\alpha_1(n,\omega), \dots, \alpha_d(n,\omega))$ with $\alpha_1(n,\omega)\ge \alpha_2(n,\omega) \ge \dots \alpha_d(n,\omega)> 0$ and $U_n(\omega), V_n(\omega)$ are orthogonal matrices. 

It is easy to see 
\[
\|S_{n}^*(\omega)\|_2 = \sup_{\|x\|_2=1}\|U_n(\omega) D_n(\omega) V_n(\omega)x\|_2 = \|D_n(\omega)\|_2 = \alpha_1(n,\omega).
\]
Thus 
\[
\dfrac{S_{n}^*(\omega)}{\|S_{n}^*(\omega)\|_2} = V_n^*(\omega)\operatorname{diag}\left(1, \dfrac{\alpha_2(n,\omega)}{\alpha_1(n,\omega)}, \dots, \dfrac{\alpha_d(n,\omega)}{\alpha_1(n,\omega)}\right) U_n^*(\omega).
\]
As the space $O(d)$ of orthogonal matrices is compact, and each $0\le \alpha_i(n,\omega)/\alpha_1(n,\omega) \le 1$, let $V_\infty(\omega), U_\infty(\omega)$ and $\beta_i(\omega)$ be limit points of $V_n(\omega), U_n(\omega)$ and $\alpha_i(\omega)/\alpha_1(\omega)$ respectively (where $2\le i\le d$). Now obtain
\[
M(\omega) := V_\infty^*(\omega)\operatorname{diag}(1, \beta_2(\omega), \dots, \beta_d(\omega)) U_\infty^*(\omega)
\]
as a limit point of $S_{n}^*(\omega)/\|S_{n}^*(\omega)\|_2$. As $\mu$ is strongly irreducible and contracting, a limit point $M(\omega)$ almost surely is a rank one matrix with range $W(\omega)$ (see e.g. \cite[Theorem III.3.1]{Bougerol}). That is, $\beta_2(\omega) = \dots = \beta_d(\omega) = 0$ almost surely. As $\beta_i(\omega)$ is an arbitrary limit point of $\alpha_i(n,\omega)/\alpha_1(n,\omega)$, we have
\begin{equation}\label{eq:aa}
\lim_{n\to\infty}\dfrac{\alpha_i(n,\omega)}{\alpha_1(n,\omega)} = 0, \ \text{ a.s. } \ (2\le i\le d) .
\end{equation}
Moreover, $W(\omega)$ is spanned by the direction $\widetilde{Z}(\omega):= [V_\infty^*(\omega)e_1]$, not depending on the choice of limit points $V_\infty^*(\omega)$, where $e_1, \dots, e_d$ denotes the standard basis of $\mathbb{R}^d$.

Then, for vector $x(\omega)\not = 0$, we have
\begin{align*}
\dfrac{\|S_{n}(\omega)x(\omega)\|_2^2}{\|S_{n}(\omega)\|_2^2} & = \dfrac{\|D_n(\omega)V_n(\omega)x(\omega)\|_2^2}{\alpha_1^2(n, \omega)} \\
& = \dfrac{\langle V_n(\omega)x(\omega), D_n^*(\omega)D_n(\omega)V_n(\omega)x(\omega) \rangle}{\alpha_1^2(n,\omega)} \\
& = \sum_{i = 1}^d \dfrac{\alpha_i^2(n,\omega)}{\alpha_1^2(n, \omega)}\langle V_n(\omega)x(\omega), e_i\rangle^2.
\end{align*}
where in the final equality we write $V_n(\omega)x(\omega) = \sum_{i = 1}^d \langle V_n(\omega)x(\omega), e_i \rangle e_i$. Consequently, by \eqref{eq:aa} almost surely we have
\[
\lim_{n\to\infty} \dfrac{\|S_{n}(\omega)x(\omega)\|_2^2}{\|S_{n}(\omega)\|_2^2} = \sum_{i=1}^d \beta_i^2(\omega)\langle x(\omega), V_\infty^*(\omega)e_i \rangle^2 = \langle x(\omega), \widetilde{Z}(\omega) \rangle^2.
\]
Finally, the convergence is uniform in $x(\omega)$ because the statement is equivalent to the one assuming $\|x(\omega)\| = 1$ by normalisation on both sides and by the fact that the unit circle $\mathbb{S}^{d-1}$ is compact. Alternatively, we can see from the proof above that the essence of the convergence comes from $V_n(\omega)\to V_\infty(\omega)$, which is irrelevant with $x(\omega)$.
\end{proof}

With almost exactly the same proof, we have the following corollary. Note the right-hand side of the equation is well-defined since $\widetilde{Z}(\omega)$ appear both on the numerator and the denominator, thus the value does not depend on the choice of representatives.

\begin{cor}\label{cor:sss}
Assume $\mu$ is strongly irreducible and contracting, then for any non-zero random vectors $x(\omega)$ and $y(\omega)$ with $\mathbb{P}(\langle \widetilde{Z}(\omega), x(\omega)\rangle = 0) = 0$, we have
\[
\lim_{n\to\infty}\dfrac{\langle S_{n}(\omega)x(\omega), S_{n}(\omega)y(\omega) \rangle}{\|S_{n}(\omega)x(\omega)\|^2_2} = \dfrac{\langle \widetilde{Z}(\omega), y(\omega)\rangle}{\langle \widetilde{Z}(\omega), x(\omega)\rangle} \quad \text{ a.s.},
\]
where $\widetilde{Z}(\omega)$ is the random direction spanning the range of any limit point of $S_{n}^*(\omega)/\|S_{n}^*(\omega)\|$. The convergence is uniform in $x(\omega)\not = 0$ and $y(\omega)\not = 0$.
\end{cor}

The following is a collection of properties of $Z(\omega), \widetilde{Z}(\omega)$ for later use. Similar to the discussion before \autoref{lm:ss}, $\|S_{n}(\omega)Z(\omega)\|$ and $|\langle \widetilde{Z}(\omega), Z(\omega) \rangle|$ are defined by choosing normalised vectors as representatives, and are obviously well-defined. One ought to distinguish $S_{n}(\omega)Z(\omega)$ from $S_{n}(\omega)\cdot Z(\omega)$, where the former denotes $\pm S_{n}(\omega)z(\omega)\in\mathbb{R}^d$ (well-defined when there is no ambiguity of signs) after choosing a normalised representative $z(\omega)\in\mathbb{S}^{d-1}$ for $Z(\omega)\in \mathcal{P}(\mathbb{R}^d)$, and the latter denotes the action of $S_{n}(\omega)$ on $Z(\omega)\in \mathcal{P}(\mathbb{R}^d)$.

\begin{lemma} \label{lm:zlm}
Assume $\mu$ is strongly irreducible and contracting and satisfies (FE1) and (FE3), then
\begin{enumerate}[label= (\alph*)]
\item \label{zlm:tz}
$A_1^*(\omega)\cdot\widetilde{Z}(\sigma\omega) = \widetilde{Z}(\omega)$;

\item \label{zlm:z}
$A_1(\omega)\cdot Z(\omega) = Z(\sigma\omega)$;

\item \label{zlm:nu}
$\nu:=\displaystyle\int \delta_{Z(\omega)}\mathbb{P}(\operatorname{d}\omega)$ is a $\mu$-invariant measure on the projective space $\mathcal{P}(\mathbb{R}^d)$;

\item \label{zlm:nust}
$\nu^*:=\displaystyle\int \delta_{\widetilde{Z}(\omega)}\mathbb{P}(\operatorname{d}\omega)$ is a $\mu^*$-invariant measure on the projective space $\mathcal{P}(\mathbb{R}^d)$;

\item \label{zlm:le}
Let $\gamma$ be the top Lyapunov exponent defined by \eqref{lim:le}, then
\begin{align*}
\lim_{n\to\infty} \dfrac{1}{n}\mathbb{E}\left[\log\|S_{n}(\omega)Z(\omega)\|\right] &= \gamma,\\
\lim_{n\to\infty} \dfrac{1}{n}\log\|S_{n}(\omega)Z(\omega)\| &= \gamma, \quad \text{ a.s.};
\end{align*}

\item \label{zlm:nzero}
$|\langle \widetilde{Z}(\omega), Z(\omega)\rangle|\not = 0\ \text{ a.s.}$;

\item \label{zlm:fin.inv}
$\displaystyle\sup_{\omega\in\Omega}\left[\dfrac{1}{|\langle \widetilde{Z}(\omega), Z(\omega)\rangle|}\right] < \infty$.
\end{enumerate}
\end{lemma}

\begin{proof}
\ref{zlm:tz} and \ref{zlm:z} follow immediately from the definitions of $\widetilde{Z}(\omega)$ and $Z(\omega)$. To prove \ref{zlm:nu}, note $A_1(\omega)$ and $Z(\omega)$ are independent, and for any bounded Borel function $f$ on $\mathcal{P}(\mathbb{R}^d)$, we have 
\begin{align*}
\iint f(A\cdot [x])\operatorname{d}\mu(A)\operatorname{d}\nu([x]) & = \int f(A_1(\omega)\cdot Z(\omega))\mathbb{P}(\operatorname{d}\omega)  & \text{(by definition of $\nu$)}\\
& = \int f(Z(\sigma\omega))\mathbb{P}(\operatorname{d}\omega) & \text{(by \ref{zlm:z})}\\
&= \int f(Z(\omega)) \mathbb{P}(\operatorname{d}\omega) & \text{(as $\sigma$ is measure-preserving)} \\
& = \int f([x])\operatorname{d}\nu([x]). & \text{(by definition of $\nu$)}
\end{align*}
Therefore $\nu$ is a $\mu$-invariant measure. \ref{zlm:nust} follows similarly.

For \ref{zlm:le}, we first notice
\[
\dfrac{1}{n}\log\|S_{n}(\omega)Z(\omega)\| = \dfrac{1}{n}\sum_{i=1}^n \log\|A_1(\sigma^{i-1}\omega)Z(\sigma^{i-1}\omega)\| \to \mathbb{E}[\log\|A_1(\omega)Z(\omega)\|]\quad \text{ a.s.} \ (n\to\infty),
\]
by Birkhoff's ergodic theorem assuming (FE1). Then notice that
\[
\mathbb{E}[\log\|A_1(\omega)Z(\omega)\|] = \iint\log\|A[x]\|\operatorname{d}\mu(A)\operatorname{d}\nu([x]) = \gamma,
\]
where the second equality is Furstenberg and Kifer's projective space formula for the Lyapunov exponent (see \cite[Theorem 2.1]{Furstenberg3} or \cite[Proposition III.7.2]{Bougerol}).

For \ref{zlm:nzero}, note if $|\langle \widetilde{Z}(\omega), Z(\omega) \rangle| = 0$, then by Oseledets' multiplicative ergodic theorem or \cite[Corollary VI.1.7]{Bougerol},
\[
\limsup_{n\to\infty}\dfrac{1}{n}\log\|S_{n}(\omega)Z(\omega)\| \le \gamma_2 < \gamma\quad \text{ a.s.},
\]
where $\gamma_2$ is the second Lyapunov exponent and $\gamma_2< \gamma$ follows from $\mu$ being contracting \cite[Proposition III.6.2]{Bougerol}. This contradicts with \ref{zlm:le}.

For \ref{zlm:fin.inv}, \autoref{lm:ss} gives
\[
\lim_{n\to\infty}\dfrac{\|S_{n}(\omega)\|_2}{\|S_{n}(\omega)Z(\omega)\|_2} = \dfrac{1}{|\langle \widetilde{Z}(\omega), Z(\omega)\rangle|} \quad \text{ a.s.}.
\]
Therefore, for $\epsilon > 0$, there exists an integer $N_1 > 0$ such that when $n > N_1$, we have
\[
\dfrac{1}{|\langle \widetilde{Z}(\omega), Z(\omega)\rangle|} \le \dfrac{\|S_{n}(\omega)\|_2}{\|S_{n}(\omega)Z(\omega)\|_2} + \epsilon\quad \text{ a.s.}.
\]
Also by \ref{zlm:le} and $\lim_{n\to\infty}(1/n)\log\|S_{n}(\omega)\|=\gamma\ \text{ a.s.}$, there exists an integer $N_2 > 0$ such that when $n > N_2$, we have
\[
\|S_{n}(\omega)\| \le e^{n(\gamma+\epsilon)},\quad \|S_{n}(\omega)Z(\omega)\| \ge e^{n(\gamma-\epsilon)}\quad \text{ a.s.}.
\]
Fix an $n > \max\{N_1, N_2\}$, then
\begin{equation}\label{eq:uni.bnd.zz}
\dfrac{1}{|\langle \widetilde{Z}(\omega), Z(\omega)\rangle|} \le e^{2n\epsilon}+\epsilon < \infty\quad \text{ a.s.}.
\end{equation}
\end{proof}

\begin{cor}\label{cor:ss}
Let $\mu$ be strongly irreducible and contracting, and satisfy (FE1) and (FE3), then for any fixed $n\ge 1$,
\[
\sup_m\dfrac{\|S_{m}(\sigma^{n}\omega)\|_2}{\|S_{m+n}(\omega)\|_2} < \infty \quad \text{ a.s.}.
\]
\end{cor}
\begin{proof}
By \autoref{lm:ss} and \autoref{lm:zlm} \ref{zlm:tz},
\begin{align*}
\lim_{m\to\infty}\dfrac{\|S_{m+n}(\omega)Z(\omega)\|_2}{\|S_{m}(\sigma^{n}\omega)\|_2} 
& = |\langle \widetilde{Z}(\sigma^{n}\omega), S_{n}(\omega)Z(\omega) \rangle| \\
& = |\langle S_{n}^*(\omega)\widetilde{Z}(\sigma^{n}\omega), Z(\omega) \rangle| \\
& = \|S_{n}(\omega)\widetilde{Z}(\sigma^{n}\omega)\|_2\cdot|\langle \widetilde{Z}(\omega), Z(\omega)\rangle|\quad \text{ a.s.}.
\end{align*}
By \autoref{lm:zlm} \ref{zlm:nzero}, this limit is almost surely non-zero. Thus there exists $\delta>0$, such that when $m$ is sufficiently large, we have
\[
\dfrac{\|S_{m+n}(\omega)\|_2}{\|S_{m}(\sigma^{n}\omega)\|_2} \ge \dfrac{\|S_{m+n}(\omega)Z(\omega)\|_2}{\|S_{m}(\sigma^{n}\omega)\|_2} > \delta\quad \text{ a.s.}.
\]
The conclusion follows.
\end{proof}

Finally, we define similar random directions for the dual problem. On the same probability space $(\Omega,\mathcal{A},\mathbb{P}, \sigma)$, define $\mathfrak{S}_n(\omega):= A_n^*(\omega)\cdots. A_1^*(\omega)$, which is the partial product obtained from the i.i.d. sequence $\{A_n^*(\omega)\}_{n\ge 1}$ following the probability measure $\mu^*$. It is not hard to see that $\mathfrak{S}_n^*(\omega)\sim_dS_n(\omega)$ and $\mathfrak{S}_n(\omega)\sim_dS_{n}^*(\omega)$. This depends on the assumption that the sequence is i.i.d..

Just like the beginning of the section, one can define the random direction $\widetilde{Z}^*(\omega)$ from the product $\mathfrak{S}_{n}^*(\omega)$ as well as the random direction $Z^*(\omega)$ from the product $\mathfrak{S}_{n}(\sigma^{-n}\omega)$. Since $\mathfrak{S}_{n}^*(\omega)\sim_d S_{n}(\sigma^{-n}\omega)$, we have $\widetilde{Z}^*(\omega)\sim_d Z(\omega)$ by definition. Similarly, $Z^*(\omega)\sim_d \widetilde{Z}(\omega)$.

\section{Asymptotic Behaviour of $\varphi_n(\omega, Z(\omega))$}
\label{sec:varphi.z}

In this section, we plan to investigate the asymptotic behaviour of $\varphi_n(\omega, Z(\omega))$ specifying the matrix norm to be $\ell_2$-norm. Let $Z(\omega), \widetilde{Z}(\omega)$ be random directions defined in \Cref{sec:z}. Define
\begin{equation}\label{defn:xi}
\xi(\omega) := \dfrac{\langle \widetilde{Z}(\sigma\omega), B_1(\omega)Z(\omega) \rangle}{\langle \widetilde{Z}(\sigma\omega), A_1(\omega)Z(\omega)\rangle}.
\end{equation}
We first notice that since $\widetilde{Z}(\sigma\omega), Z(\omega)\in \mathcal{P}(\mathbb{R}^d)$ both simultaneously appear on the numerator and the denominator, $\xi(\omega)$ does not depend on the choice of the representatives.
Next we notice that the denominator (under absolute value) can be written as
\[
|\langle \widetilde{Z}(\sigma\omega), A_1(\omega)Z(\omega) \rangle| = 
|\langle A_1^*(\omega)\widetilde{Z}(\sigma\omega), Z(\omega) \rangle| = 
\|A_1^*(\omega)\widetilde{Z}(\sigma\omega)\|\cdot |\langle \widetilde{Z}(\omega),Z(\omega) \rangle|,
\]
by \autoref{lm:zlm} \ref{zlm:tz}. Note by \autoref{lm:zlm} \ref{zlm:nzero}, $|\langle \widetilde{Z}(\omega), Z(\omega)\rangle| \not = 0$ almost surely. Thus $\xi(\omega)$ is well-defined almost surely. Since $\widetilde{Z}(\sigma\omega), A(\omega)$ and $Z(\omega)$ are independent, by \autoref{lm:zlm} \ref{zlm:nu} \ref{zlm:nust}, we also have an integral formula for $\mathbb{E}[\xi(\omega)]$ as follows
\begin{equation}\label{eq:xi.int}
\mathbb{E}[\xi(\omega)] = \iiint \dfrac{\langle [y], B[x]\rangle}{\langle [y], A[x]\rangle}\operatorname{d}\mu(A, B)\operatorname{d}\nu([x])\operatorname{d}\nu^*([y]),
\end{equation}
where $\nu$ is a $\mu$-invariant measure and $\nu^*$ is a $\mu^*$-invariant measure. We remark here again that the integrand means $\langle y, Bx\rangle/ \langle y, Ax\rangle$ where $x, y$ are representatives of $[x], [y]\in \mathcal{P}(\mathbb{R}^d)$.

\begin{lemma}\label{lm:fe2}
If $\mathbb{E}[\|B_1(\omega)A_1^{-1}(\omega)\|]<\infty$, that is, $\mu$ satisfies (FE2), then $\mathbb{E}|\xi(\omega)|<\infty$.
\end{lemma}
\begin{proof}
By \autoref{lm:zlm} \ref{zlm:z}, we have
\[
|\xi(\omega)| \le \dfrac{\|B_1(\omega)A_1^{-1}(\omega)A_1(\omega)Z(\omega)\|}{\|A_1(\omega)Z(\omega)\|}\dfrac{1}{|\langle \widetilde{Z}(\sigma\omega),Z(\sigma\omega) \rangle|}
\le  \dfrac{\|B_1(\omega)A_1^{-1}(\omega)\|}{|\langle \widetilde{Z}(\sigma\omega),Z(\sigma\omega) \rangle|} .
\]
Thus 
\[
\mathbb{E}|\xi(\omega)| \le \mathbb{E}[\|B_1(\omega)A_1^{-1}(\omega)\|] \cdot \sup_{\omega\in\Omega}\left[\dfrac{1}{|\langle \widetilde{Z}(\sigma\omega),Z(\sigma\omega) \rangle|}\right] < \infty,
\]
by the assumption and \autoref{lm:zlm} \ref{zlm:fin.inv}.
\end{proof}

In the following proposition, we specify the matrix norms in the definition \eqref{defn:x} of $\varphi_n$ to be $\ell_2$-norms.
\begin{prop} \label{prop:varphi.z}
Assume $\mu$ is strongly irreducible and contracting, and satisfies (FE1) - (FE3). Then 
\[
\dfrac{\varphi_n(\omega, Z(\omega))}{n} \to \left|\mathbb{E}[\xi(\omega)]\right| \quad \text{ a.s.} \quad (n\to\infty)
\]
\end{prop}
\begin{proof}
By \autoref{lm:ss}, we have
\begin{equation}\label{eq:stss.lim}
\lim_{n\to\infty}\dfrac{\|S_{n}(\sigma^{m}\omega)T_{m}(\omega)Z(\omega)\|_2}{\|S_{n}(\sigma^{m}\omega)S_{m}(\omega)Z(\omega)\|_2} 
=\left| \dfrac{\langle \widetilde{Z}(\sigma^{m}\omega), T_{m}(\omega)Z(\omega)\rangle}{\langle \widetilde{Z}(\sigma^{m}\omega), S_{m}(\omega)Z(\omega)\rangle}\right| \quad \text{ a.s.}.
\end{equation}
Since the convergence in \autoref{lm:ss} is uniform for all $x \not = 0$, the convergence in \eqref{eq:stss.lim} is also uniform for $m \ge 1$, that is, for $\epsilon > 0$, there exists an integer $N_1:= N_1(\epsilon)>0$, such that when $n > N_1$,
\begin{equation}\label{eq:stss.ineq}
\left|\dfrac{\langle \widetilde{Z}(\sigma^{m}\omega), T_{m}(\omega)Z(\omega)\rangle}{\langle \widetilde{Z}(\sigma^{m}\omega), S_{m}(\omega)Z(\omega)\rangle}\right| -\epsilon
\le \dfrac{\|S_{n}(\sigma^{m}\omega)T_{m}(\omega)Z(\omega)\|_2}{\|S_{n}(\sigma^{m}\omega)S_{m}(\omega)Z(\omega)\|_2} 
\le \left|\dfrac{\langle \widetilde{Z}(\sigma^{m}\omega), T_{m}(\omega)Z(\omega)\rangle}{\langle \widetilde{Z}(\sigma^{m}\omega), S_{m}(\omega)Z(\omega)\rangle}\right| +\epsilon.
\end{equation}

Note by \autoref{lm:zlm} \ref{zlm:tz} and \ref{zlm:z}
\begin{equation}\label{eq:et.xi}
\begin{aligned}
\dfrac{1}{m}\dfrac{\langle \widetilde{Z}(\sigma^{m}\omega), T_{m}(\omega)Z(\omega)\rangle}{\langle \widetilde{Z}(\sigma^{m}\omega), S_{m}(\omega)Z(\omega)\rangle} 
& = \dfrac{1}{m}\sum_{i=1}^m \dfrac{\langle \widetilde{Z}(\sigma^{m}\omega), S_{m-i}(\sigma^{i}\omega)B_i(\omega)S_{i-1}(\omega)Z(\omega)\rangle}{\langle \widetilde{Z}(\sigma^{m}\omega), S_{m-i}(\sigma^{i}\omega)A_i(\omega)S_{i-1}(\omega)Z(\omega)\rangle}\\
& = \dfrac{1}{m}\sum_{i=1}^m \dfrac{\langle S_{m-i}^*(\sigma^{i}\omega)\widetilde{Z}(\sigma^{m}\omega), B_i(\omega)Z(\sigma^{i-1}\omega)\rangle}{\langle S_{m-i}^*(\sigma^{i}\omega)\widetilde{Z}(\sigma^{m}\omega), A_{i}(\omega)Z(\sigma^{i-1}\omega)\rangle} \\
& = \dfrac{1}{m}\sum_{i=1}^m \dfrac{\langle \widetilde{Z}(\sigma^{i}\omega), B_1(\sigma^{i-1}\omega)Z(\sigma^{i-1}\omega) \rangle}{\langle \widetilde{Z}(\sigma^{i}\omega), A_1(\sigma^{i-1}\omega)Z(\sigma^{i-1}\omega) \rangle}\\
& \to \mathbb{E}[\xi(\omega)] \quad \text{ a.s.},
\end{aligned}
\end{equation}
as $m \to\infty$ by Birkhoff's ergodic theorem and \autoref{lm:fe2}. 
Consequently, for $\epsilon > 0$, one can find integer $M:= M(\epsilon) > 0$ such that when $m > M$,
\begin{equation}\label{eq:xi.bnd}
m\left(\left|\mathbb{E}[\xi(\omega)]\right| -\epsilon\right)
\le \left|\dfrac{\langle \widetilde{Z}(\sigma^{m}\omega), T_{m}(\omega)Z(\omega)\rangle}{\langle \widetilde{Z}(\sigma^{m}\omega), S_{m}(\omega)Z(\omega)\rangle}\right| 
\le m\left(\left|\mathbb{E}[\xi(\omega)]\right| +\epsilon\right)\quad \text{ a.s.}.
\end{equation}
Consequently, by \eqref{eq:t.ind} \eqref{eq:stss.ineq} \eqref{eq:xi.bnd},
\begin{align*}
\varphi_{n+m}(\omega, Z(\omega)) &= \dfrac{\|S_{n}(\sigma^{m}\omega)T_{m}(\omega)Z(\omega) + T_{n}(\sigma^{m}\omega)S_{m}(\omega)Z(\omega) \|_2}{\|S_{n+m}(\omega)Z(\omega)\|_2}\\
& \ge \dfrac{\|S_{n}(\sigma^{m}\omega)T_{m}(\omega)Z(\omega)\|_2}{\|S_{n+m}(\omega)Z(\omega)\|_2} - \dfrac{\|T_{n}(\sigma^{m}\omega)Z(\sigma^{m}\omega)\|_2}{\|S_{n}(\sigma^{m}\omega)Z(\sigma^{m}\omega)\|_2} \\
& \ge m\left(\left|\mathbb{E}[\xi(\omega)]\right| -\epsilon\right)-\epsilon - \dfrac{C_2(n)}{C_0(n)} \quad \text{ a.s.},
\end{align*}
where $C_0(n)>0$ and $C_2(n)$ are given by \autoref{lm:st.bnd} when $n>N_2$ for some integer $N_2$ assuming (FE1) - (FE3). By fixing an $n>\max\{N_1, N_2\}$ and letting $m\to\infty$, we have 
\begin{equation}\label{eq:liminf.varphi}
\liminf_{m\to\infty}\dfrac{\varphi_{n+m}(\omega, Z(\omega))}{m} \ge \left| \mathbb{E}[\xi(\omega)]\right| - \epsilon \quad \text{ a.s.}.
\end{equation}
Similarly, by \eqref{eq:stss.ineq} and \eqref{eq:xi.bnd}, when $n$ is sufficiently large,
\begin{align*}
\varphi_{n+m}(\omega, Z(\omega)) 
& \le \dfrac{\|S_{n}(\sigma^{m}\omega)T_{m}(\omega)Z(\omega)\|_2}{\|S_{n+m}(\omega)Z(\omega)\|_2} + \dfrac{\|T_{n}(\sigma^{m}\omega)Z(\sigma^{m}\omega)\|_2}{\|S_{n}(\sigma^{m}\omega)Z(\sigma^{m}\omega)\|_2} \\
& \le m\left(\left|\mathbb{E}[\xi(\omega)]\right| +\epsilon\right)+\epsilon + \dfrac{C_2(n)}{C_0(n)} \quad \text{ a.s.}.
\end{align*}
Thus 
\begin{equation}\label{eq:limsup.varphi}
\limsup_{m\to\infty}\dfrac{\varphi_{n+m}(\omega, Z(\omega))}{m} \le \left|\mathbb{E}[\xi(\omega)]\right|+\epsilon \quad \text{ a.s.}.
\end{equation}
Since $\epsilon>0$ is arbitrary, the conclusion follows by \eqref{eq:liminf.varphi} and \eqref{eq:limsup.varphi}.
\end{proof}

\begin{cor} \label{cor:t.le}
Assume $\mu$ is strongly irreducible and contracting, and satisfies (FE1) - (FE3). If $\mathbb{E}[\xi(\omega)] \not = 0$, then
\[
\lim_{n\to\infty}\dfrac{1}{n}\log\|T_{n}(\omega)Z(\omega)\| = \gamma \quad \text{ a.s.}.
\]
\end{cor}
\begin{proof}
By \autoref{prop:varphi.z}, for $\epsilon>0$, when $n$ is sufficiently large, we have
\[
n (|\mathbb{E}[\xi(\omega)]| - \epsilon)\|S_{n}(\omega)Z(\omega)\|_2 \le \|T_{n}(\omega)Z(\omega)\|_2 \le n (|\mathbb{E}[\xi(\omega)]| + \epsilon)\|S_{n}(\omega)Z(\omega)\|_2.
\]
When $\mathbb{E}[\xi(\omega)]\not = 0$, we have $|\mathbb{E}[\xi(\omega)]|-\epsilon>0$ when $\epsilon$ is sufficiently small. Therefore, the conclusion follows from \autoref{lm:zlm} \ref{zlm:le} and the equivalence of $\ell_2$-norms with general matrix norms.
\end{proof}
We remark here that $\mathbb{E}[\xi(\omega)]\not = 0$ is needed for the conclusion to hold. A trivial example is when $B(\omega)\equiv 0$. In this case $T_{n}(\omega) \equiv 0$ and the limit in the corollary would be $-\infty$. Another less trivial example can be found in \Cref{sec:exmp} \autoref{exmp2}, in which $T_{n}(\omega)$ visits $0$ infinitely often almost surely.

Finally, note that we can define the corresponding $\xi^*(\omega)$ for the dual problem given by the i.i.d. sequence $\{A_n^*(\omega)\}_{n\ge1}$ following the probability measure $\mu^*$ and $B_n^*(\omega):= (B_n(\omega))^*$. By \eqref{defn:xi} and the discussion at the end of \Cref{sec:z}, we should let
\[
\xi^*(\omega) := \dfrac{\langle \widetilde{Z}^*(\sigma\omega), B_1^*(\omega)Z^*(\omega)\rangle}{\langle \widetilde{Z}^*(\sigma\omega), A_1^*(\omega)Z^*(\omega)\rangle} =  \dfrac{\langle B_1(\omega)\widetilde{Z}^*(\sigma\omega), Z^*(\omega)\rangle}{\langle A_1(\omega)\widetilde{Z}^*(\sigma\omega), Z^*(\omega)\rangle}.
\]
Note that $(Z^*(\omega), A(\omega), \widetilde{Z}^*(\sigma\omega))\sim_d(\widetilde{Z}(\sigma\omega), A(\omega), Z(\omega))$ by the definitions of these random directions, we have that $\xi(\omega)$ and $\xi^*(\omega)$ also follow the same law. Here we remind the readers again that this depends on the random products being i.i.d..

Let $\mathfrak{S}_{n}(\omega) := A_n^*(\omega)\cdots A_1^*(\omega)$ be the product in the dual problem as in last section, also inductively define $\mathfrak{T}_{n}(\omega):= A_n^*(\omega)\mathfrak{T}_{n-1}(\omega) + B_n^*(\omega)\mathfrak{S}_{n-1}(\omega)$, that is, $\mathfrak{T}_n(\omega)= \mathfrak{D}_n(\omega, A^*, B^*)$ is the first derived product associated to $(A^*, B^*)$. When $\mu$ is strongly irreducible and contracting, $\mu^*$ is also strongly irreducible and contracting. Thus by \autoref{prop:varphi.z}, we have
\[
\lim_{n\to\infty}\dfrac{1}{n}\dfrac{\|\mathfrak{T}_{n}(\omega)Z^*(\omega)\|_2}{\|\mathfrak{S}_{n}(\omega)Z^*(\omega)\|_2} = |\mathbb{E}[\xi^*(\omega)]| = |\mathbb{E}[\xi(\omega)]|\quad \text{ a.s.}.
\]
Note $(\mathfrak{S}_{n}(\omega), Z^*(\omega))\sim_d ( S_{n}^*(\omega), \widetilde{Z}(\sigma^{n}\omega))$, thus 
\[
\dfrac{\|\mathfrak{T}_{n}(\omega)Z^*(\omega)\|_2}{\|\mathfrak{S}_{n}(\omega)Z^*(\omega)\|_2}\sim_d\dfrac{\|T_{n}^*(\omega)\widetilde{Z}(\sigma^{n}\omega)\|_2}{\|S_{n}^*(\omega)\widetilde{Z}(\sigma^{n}\omega)\|_2},
\]
we have the following corollary immediately
\begin{cor}\label{cor:varphi.tz}
When $\mu$ is strongly irreducible and contracting, and satisfies (FE1) - (FE3). Then
\[
\lim_{n\to\infty}\dfrac{1}{n}\dfrac{\|T_{n}^*(\omega)\widetilde{Z}(\sigma^{n}\omega)\|_2}{\|S_{n}^*(\omega)\widetilde{Z}(\sigma^{n}\omega)\|_2} = |\mathbb{E}[\xi(\omega)]|\quad \text{ a.s.}.
\]
\end{cor}

\section{Inductive Formulas}
\label{sec:ind}
In this section, we start with giving some inductive formulas involving $x_m, y_m, \varphi_m$ and $\psi_m$ given a random vector $x:= x(\omega)\not = 0$. They will be frequently used in later sections. After that, we give an easy lemma about the well-definedness of $y_m$ by the information provided from $\varphi_m$. At the end of the section, we will digress to give some interesting results about the visits of $T_{n}(\omega)$ at $0$ in broad generalities, including a zero-one law as well as a surprising criterion using traces. These interesting results will not be used in the later part of the article.

\begin{lemma}\label{ind.lm}
Given a random vector $x:= x(\omega)\not= 0$, then for integers $n, m\ge 0$, we have
\begin{enumerate}[label = (\alph*)]
\item\label{lm:s}
$\|S_{n+m}(\omega)x\| = \|S_{n}(\sigma^m\omega)x_m\|\|S_{m}(\omega)x\|$; 

\item\label{lm:t} $\|T_{n+m}(\omega)x\| = 
\begin{cases}
\|S_{n}(\sigma^m\omega)y_m\varphi_m + T_{n}(\sigma^m\omega)x_m\|\|S_{m}(\omega)x\| & \varphi_m\not = 0,\\
\|T_{n}(\sigma^m\omega)x_m\|\|S_{m}(\omega)x\| & \varphi_m = 0;
\end{cases}$

\item\label{lm:varphi}
$\varphi_{n+m}(\omega,x) = 
\begin{cases}
\|S_{n}(\sigma^m\omega)y_m\varphi_m + T_{n}(\sigma^m\omega)x_m\|/\|S_{n}(\sigma^m\omega)x_m\| & \varphi_m \not= 0, \\
\|T_{n}(\sigma^m\omega)x_m\|/\|S_{n}(\sigma^m\omega)x_m\| & \varphi_m = 0;
\end{cases}$

\item\label{lm:x} 
$x_{n+m}(\omega, x) = \overline{S_{n}(\sigma^m\omega)x_m} = x_n(\sigma^m\omega, x_m(\omega, x))$; 

\item\label{lm:y} if $\varphi_{n+m} \not=0$, then
$y_{n+m}(\omega, x) = 
\begin{cases}
\overline{S_{n}(\sigma^m\omega)y_m\varphi_m + T_{n}(\sigma^m\omega)x_m} & \varphi_m\not=0 \\
\overline{T_{n}(\sigma^m\omega)x_m} & \varphi_m = 0.
\end{cases}$
\end{enumerate}
\end{lemma}
\begin{proof}
Note $y_m$ is well-defined only if $\|T_{m}(\omega)x\|\not= 0$, which is equivalent to $\varphi_m\not = 0$. Here we only prove the results when $\varphi_m \not = 0$. The case when $\varphi_m = 0$ can be proved similarly. 
\ref{lm:s} follows easily from $S_{n+m}(\omega)x = S_{n}(\sigma^m\omega)S_{m}(\omega)x$. By \eqref{eq:t.ind}, we have
\begin{equation}\label{eq:t.norm}
\begin{aligned}
T_{n+m}(\omega)x & = S_{n}(\sigma^m\omega)T_{m}(\omega)x + T_{n}(\sigma^m\omega)S_{m}(\omega)x \\
& = \left(S_{n}(\sigma^m\omega)\dfrac{T_{m}(\omega)x}{\|T_{m}(\omega)x\|}\dfrac{\|T_{m}(\omega)x\|}{\|S_{m}(\omega)x\|} + T_{n}(\sigma^m\omega)\dfrac{S_{m}(\omega)x}{\|S_{m}(\omega)x\|}\right)\|S_{m}(\omega)x\| \\
& = (S_{n}(\sigma^m\omega)y_m\varphi_m + T_{n}(\sigma^m\omega)x_m)\|S_{m}(\omega)x\|.
\end{aligned}
\end{equation}
This proves \ref{lm:t}. Then \ref{lm:varphi} and \ref{lm:x} follow from the definitions of $\varphi_{n+m}$ and $x_{n+m}$ respectively together with the formulas \ref{lm:s} and \ref{lm:t}. Finally, \ref{lm:y} follows from the definition of $y_{n+m}$ and the computation of $T_{n+m}(\omega)x$ above.
\end{proof}

When specifying the norm to be $\ell_2$-norm, one can define $\psi_m$ as in \eqref{defn:psi} by
\[
\psi_m(\omega, x):= \dfrac{\langle S_{m}(\omega)x, T_{m}(\omega)x\rangle}{\|S_{m}(\omega)x\|_2^2}.
\]
And the following lemma follows immediately from \eqref{eq:t.norm}.
\begin{lemma}\label{ind.lm:psi}
Given a random vector $x(\omega)\not = 0$, then for integers $n, m\ge 0$ with $\varphi_m\not = 0$, we have
\[
\psi_{n+m}(\omega, x) = \dfrac{\langle S_{n}(\sigma^{m}\omega)x_m, S_{n}(\sigma^{m}\omega)y_m \rangle}{\|S_{n}(\omega)x_m\|^2_2}\varphi_m + \psi_n(\omega, x_m).
\]
\end{lemma}

As we have remarked after \eqref{defn:x}, $y_m$ is well-defined if and only if $\|T_{m}(\omega)x\|\not=0$. The following lemma gives sufficient information of the well-definedness of $y_m$ for the purpose of the paper.
\begin{lemma}\label{lm:wdef.y}
Given a random vector $x:= x(\omega)\not = 0$, if $\varphi_n(\omega, x)\to\infty$ almost surely as $n\to\infty$, then 
\[
\mathbb{P}(\|T_{n}(\omega)x\| = 0 \ \text{i.o.}) = 0.
\]
In particular, the sequence $\{y_n\}_{n \ge 1}$ defined by \eqref{defn:x} is well-defined.
\end{lemma}
\begin{proof}
Note $\|T_{n}(\omega)x\| = 0$ if and only if $\varphi_n(\omega, x) = 0$. If $\|T_{n}(\omega)x\|$ visits $0$ infinitely often, then $\varphi_n(\omega, x)$ visits $0$ infinitely often, which contradicts to $\varphi_n(\omega, x)\to\infty$ almost surely.
\end{proof}

With the lemma and \autoref{prop:varphi.z}, we know that if $\mathbb{E}[\xi(\omega)]\not = 0$, then the corresponding $y_m(\omega, Z(\omega))$ for $x = Z(\omega)$ is almost surely well-defined for large $m$. Later on in this article, we will prove $\varphi_n(\omega, x)/n$ converges to a non-zero limit almost surely if $\mathbb{E}[\xi(\omega)]\not = 0$ in \autoref{thm:varphi.x}. This implies $y_m(\omega, x)$ is well-defined for a given deterministic $x\not = 0$.

For the rest of the section, we give two results which are not used in the later part of the article but are interesting on their own. We start with a zero-one law for $T_{n}(\omega) = 0$ infinitely often. Define $\tau_1(\omega) := \inf\{n \ge 1: T_{n}(\omega) = 0\}$, and $\tau_i(\omega) := \inf\{n\ge\tau_{i-1}(\omega): T_{n}(\omega) = 0\}$.
\begin{prop} If $\mathbb{P}(\tau_1(\omega) < \infty) = 1$ then $\mathbb{P}(T_{n}(\omega) = 0 \ \text{i.o.}) = 1$, otherwise $\mathbb{P}(T_{n}(\omega) = 0 \ \text{i.o.}) = 0.$
\end{prop}
\begin{proof}
Noting $T_{\tau_1}(\omega) = 0$ and applying \eqref{eq:t.formula} gives
\begin{align*}
&\mathbb{P}(\tau_2(\omega) - \tau_1(\omega) = k | \tau_1(\omega)<\infty)\\
=& \mathbb{P}(T_{\tau_1+k}(\omega) = 0,\ T_{\tau_1+i}(\omega)\not = 0 \text{ for } 0 < i < k | \tau_1(\omega)<\infty) \\
=& \mathbb{P}(T_{k}(\sigma^{\tau_1}\omega)S_{\tau_1}(\omega) = 0, \ T_i(\sigma^{\tau_1}\omega)S_{\tau_1}(\omega) \not= 0 \text{ for } 0 < i < k|\tau_1(\omega)<\infty)\\
=& \mathbb{P}(T_k(\sigma^{\tau_1}\omega) = 0, \ T_i(\sigma^{\tau_1}\omega) \not = 0 \text{ for } 0 < i < k |\tau_1(\omega)<\infty) \\
=& \mathbb{P}(T_k(\omega) = 0,\ T_i(\omega) \not = 0 \text{ for } 0 < i < k)\\
= & \mathbb{P}(\tau_1(\omega) = k).
\end{align*}
This shows that 
\[
\mathbb{P}(\tau_2(\omega) < \infty) = \mathbb{P}(\tau_2(\omega) < \infty|\tau_1(\omega) < \infty)\mathbb{P}(\tau_1(\omega) < \infty) = \mathbb{P}(\tau_1(\omega) <\infty)^2.
\]
Similarly, in general we have $\mathbb{P}(\tau_j(\omega) < \infty) = \mathbb{P}(\tau_1(\omega) < \infty)^j$ for $j \ge 1$. 

If $\mathbb{P}(\tau_1(\omega) < \infty) = 1$, then $\mathbb{P}(\tau_j(\omega) < \infty) = 1$ for $j \ge 1$. Thus $\mathbb{P}(T_{n}(\omega) = 0 \ \text{i.o.}) = 1$.
If $\mathbb{P}(\tau_1(\omega) < \infty) < 1$, define
\[
V(\omega) = \sum_{n=1}^\infty\mathbf{1}_{\{T_{n}(\omega) = 0\}} = \sum_{m=1}^\infty \mathbf{1}_{\{\tau_m(\omega) < \infty\}}
\]
to be the number of times $T_{n}(\omega)$ taking $0$. Then by Fubini's theorem, 
\[
\sum_{n=1}^\infty \mathbb{P}(T_{n}(\omega) = 0) = \mathbb{E}[V(\omega)]=\sum_{m=1}^\infty \mathbb{P}(\tau_1(\omega) < \infty)^m < \infty.
\]
Then $\mathbb{P}(T_{n}(\omega) = 0 \ \text{i.o.}) = 0$ by the Borel-Cantelli lemma.
\end{proof}

Next we give an easy criterion for $\mathbb{P}(T_{n}(\omega) = 0 \ \text{i.o.}) = 0$. It is based on an interesting observation that $tr_n(\omega) := \operatorname{tr}(T_{n}(\omega)S_{n}^{-1}(\omega))$ is an i.i.d. random walk, where $\operatorname{tr}(\cdot)$ denotes the matrix trace. This criterion is also easier to verify than computing $\mathbb{E}[\xi(\omega)]$.
\begin{prop}
If $\mathbb{E}\left[\operatorname{tr}(B_1(\omega)A_1^{-1}(\omega))\right]\not = 0$, then $\mathbb{P}(T_{n}(\omega) = 0 \ \text{i.o.}) = 0.$
\end{prop}
\begin{proof}
Note
\begin{align*}
T_{n+1}(\omega)S_{n+1}^{-1}(\omega) & = \left(A_{n+1}(\omega)T_{n}(\omega) + B_{n+1}(\omega)S_{n}(\omega)\right) S_{n}^{-1}(\omega) A_{n+1}^{-1}(\omega)\\
& = A_{n+1}(\omega)T_{n}(\omega)S_{n}^{-1}(\omega)A_{n+1}^{-1}(\omega) + B_{n+1}(\omega)A_{n+1}^{-1}(\omega).
\end{align*}
By taking the trace, we have
\[
tr_{n+1}(\omega) = tr_n(\omega) + \operatorname{tr}(B_{n+1}(\omega)A_{n+1}^{-1}(\omega)).
\]
This shows that $\{tr_n(\omega)\}_{n\ge 1}$ is an i.i.d. random walk with each jump following the distribution of $\operatorname{tr}(B_1(\omega)A_1^{-1}(\omega))$. When $\mathbb{E}[\operatorname{tr}(B_1(\omega)A_1^{-1}(\omega))] \not = 0$, the random walk is transient \cite[Sec 4.2]{Durrett}. Therefore $\mathbb{P}(tr_n(\omega) = 0 \ \text{i.o.}) = 0$. The conclusion follows by noticing $tr_n(\omega) = 0$ if $T_{n}(\omega) = 0$.
\end{proof}

\section{Limiting Behaviours of Angles}
\label{sec:angles}
In this section, our goal is to prove that given a random vector $x(\omega)\not = 0$, if $\varphi_n\to\infty$ almost surely as $n\to\infty$, then  $S_{n}(\omega)x(\omega)$ and $T_{n}(\omega)x(\omega)$ align to the same direction almost surely as $n\to\infty$.

For two non-zero vectors $x, y\in\mathbb{R}^d$, use $\delta(x, y)$ to denote the absolute value of sine of the angle between $x$ and $y$. That is, 
\begin{equation}\label{defn:delta}
\delta(x, y) := \left(1-\dfrac{\langle x, y \rangle^2}{\|x\|_2^2\|y\|_2^2}\right)^{1/2} = \dfrac{\|x\wedge y\|_2}{\|x\|_2\|y\|_2} \in [0, 1],
\end{equation}
where $x\wedge y \in \bigwedge^2(\mathbb{R}^2)$ the exterior power of $\mathbb{R}^d$, $\langle \cdot, \cdot \rangle$ denotes the usual inner product on $\mathbb{R}^d$ inducing the $\ell_2$-norm $\|\cdot\|_2$ of vectors, and $\|\cdot\|_2$ on $\bigwedge^2(\mathbb{R}^d)$ is given by
\[
\|x \wedge y\|_2^2 = \det\begin{pmatrix}
\langle x, x \rangle & \langle x, y \rangle \\
\langle y, x \rangle & \langle y, y \rangle
\end{pmatrix}.
\]
It is not hard to verify that $\delta$ reduces to a distance on $\mathcal{P}(\mathbb{R}^d)$ \cite[III.4.1]{Bougerol} and it is related to $\|\cdot\|_2$ on $\mathbb{S}^{d-1}$ in the sense of the following lemma.
\begin{lemma}\label{lm:delta.equiv}
For normalised vectors $\|x\|_2=\|y\|_2 = 1$, we have 
\[
\dfrac{1}{\sqrt{2}}\|x-\chi y\|_2 \le \delta(x, y) \le \|x-y\|_2,
\]
where $\chi = 1$ if $\langle x,y \rangle\ge0$, and $\chi = -1$ otherwise. 
\end{lemma}
\begin{proof}
By \eqref{defn:delta} and $\|x-\chi y\|_2^2 = 2 - 2\chi\langle x, y\rangle$ if $|\chi| =1$, we have 
\[
\delta(x, y)^2 = 1 - \langle x,y \rangle^2 = \|x-\chi y\|_2^2 \left(1 - \dfrac{\|x-\chi y\|_2^2}{4}\right).
\]
This shows $\delta(x, y)\le \|x-y\|_2$. On the other hand, if $\chi$ satisfies the assumption in the statement then $\chi\langle x, y \rangle \ge 0$, thus $\|x - \chi y\|_2^2 \le 2$. The other inequality follows.
\end{proof}

For a matrix $A\in GL(d,\mathbb{R})$, let $\wedge^2 A$ be the automorphism on $\bigwedge^2(\mathbb{R}^d)$ given by
\[
(\wedge^2 A) (x \wedge y) = Ax \wedge Ay.
\]
Define $\|\wedge^2 A\|_2:= \sup\{\|(\wedge^2 A)(w)\|_2: w\in\bigwedge^2(\mathbb{R}^d), \|w\|_2 = 1\}$.

\begin{lemma}\label{lm:delta}
For integers $n, m \ge 0$ such that $\varphi_m\not = 0,\varphi_{m+n}\not =0$, we have
\begin{align*}
\delta(x_{n+m}, y_{n+m}) \le & \delta( S_{n}(\sigma^m\omega)x_m, T_{n}(\sigma^m\omega)x_m)\dfrac{\|T_{n}(\sigma^m\omega)x_m\|_2}{\|S_{n}(\sigma^m\omega)y_m\varphi_m + T_{n}(\sigma^m\omega)x_m\|_2} \\
&+ \delta( x_m, y_m) \dfrac{\varphi_m}{\|S_{n}(\sigma^m\omega)y_m\varphi_m + T_{n}(\sigma^m\omega)x_m\|_2}\dfrac{\|\wedge^2 S_{n}(\sigma^m\omega)\|_2}{\|S_{n}(\sigma^m\omega)x_m\|_2}
\end{align*}
\end{lemma}
\begin{proof}
By \eqref{defn:delta}, we have
\[
\delta(x_{n+m}, y_{n+m}) = \dfrac{\|x_{n+m}\wedge y_{n+m}\|_2}{\|x_{n+m}\|_2\|y_{n+m}\|_2}
\]
By \autoref{ind.lm}, we have
\begin{align*}
\delta(x_{n+m}, y_{n+m}) = \dfrac{\|S_{n}(\sigma^m\omega)x_m \wedge \left( S_{n}(\sigma^m\omega)y_m\varphi_m + T_{n}(\sigma^m\omega)x_m\right)\|_2}{\|S_{n}( \sigma^m\omega)x_m\|_2 \|S_{n}(\sigma^m\omega)y_m\varphi_m + T_{n}(\sigma^m\omega)x_m\|_2}.
\end{align*}
Now focusing on the numerator,
\begin{align*}
& \|S_{n}(\sigma^m\omega)x_m \wedge ( S_{n}(\sigma^m\omega)y_m\varphi_m + T_{n}(\sigma^m\omega)x_m)\|_2 \\
\le & \|S_{n}(\sigma^m\omega)x_m \wedge S_{n}(\sigma^m\omega)y_m\varphi_m\|_2 + \|S_{n}(\sigma^m\omega)x_m \wedge T_{n}(\sigma^m\omega)x_m\|_2 \\
\le & \|\wedge^2 S_{n}(\sigma^m\omega)\|_2\|x_m\wedge y_m\|_2\varphi_m + \|S_{n}(\sigma^m\omega)x_m \wedge T_{n}(\sigma^m\omega)x_m\|_2 \\
\le &\delta(x_m, y_m)\varphi_m  \|\wedge^2 S_{n}(\sigma^m\omega)\|_2 + \delta( S_{n}(\sigma^m\omega)x_m, T_{n}(\sigma^m\omega)x_m)\|S_{n}(\sigma^m\omega)x_m\|_2\|T_{n}(\sigma^m\omega)x_m\|_2.
\end{align*}
And the conclusion follows immediately.
\end{proof}

Here is another lemma from the classical results (see e.g. \cite[Sec III.6]{Bougerol}). 
\begin{lemma}\label{lm:wedges}
Assume $\mu$ is strongly irreducible and contracting, and satisfies (FE1), then for random vectors $x(\omega), y(\omega)\in\mathbb{R}^d$ satisfying $\langle \widetilde{Z}(\omega), x(\omega)\rangle\not = 0, \langle \widetilde{Z}(\omega), y(\omega)\rangle\not = 0\ \text{ a.s.}$,
\[
\lim_{n\to\infty}\dfrac{\|\wedge^2 S_{n}(\omega)\|_2}{\|S_{n}(\omega)x(\omega)\|_2\|S_{n}(\omega)y(\omega)\|_2} = 0 \quad \text{ a.s.}.
\]
The convergence is uniform in $x(\omega), y(\omega)$.
\end{lemma}
\begin{proof}
Use the same notations as in the proof of \autoref{lm:ss}, have
\[
\|\wedge^2 S_{n}(\omega)\|_2 = \alpha_1(n,\omega)\alpha_2(n,\omega), \quad
\|S_{n}(\omega)\|_2^2 = \alpha_1^2(n,\omega).
\]
Then
\[
\lim_{n\to\infty} \dfrac{\|\wedge^2 S_{n}(\omega)\|_2}{\|S_{n}(\omega)\|_2^2} = \lim_{n\to\infty} \dfrac{\alpha_2(n,\omega)}{\alpha_1(n,\omega)} = 0 \quad \text{ a.s.},
\]
where the second equality follows from \eqref{eq:aa}. Now by the uniform convergence in \autoref{lm:ss}
\[
\dfrac{\|\wedge^2 S_{n}(\omega)\|_2}{\|S_{n}(\omega)x(\omega)\|_2\|S_{n}(\omega)y(\omega)\|_2} = \dfrac{\|\wedge^2 S_{n}(\omega)\|_2}{\|S_{n}(\omega)\|_2^2}\dfrac{\|S_{n}(\omega)\|_2}{\|S_{n}(\omega)x(\omega)\|_2}\dfrac{\|S_{n}(\omega)\|_2}{\|S_{n}(\omega)y(\omega)\|_2} \to 0 \ (n\to\infty)\quad \text{ a.s.}.
\]
\end{proof}

\begin{lemma} \label{lm:dom.varphi}
Assume $\mu$ satisfies (FE1)-(FE3) and that the sequence $\varphi_m\to\infty \ (m\to\infty)$ almost surely. Then there exists an integer $N>0$ such that whenever $n>N$, we have
\[
\lim_{m\to\infty}\dfrac{\|S_{n}(\sigma^{m}\omega)y_{m}\varphi_{m}+ T_{n}(\sigma^{m}\omega)x_{m}\|_2}{\varphi_{m}\|S_{n}(\sigma^{m}\omega)y_{m}\|_2 } = 1\quad \text{ a.s.}.
\]
\end{lemma}
\begin{proof}
First notice that as $\varphi_m\to\infty$, by \autoref{lm:wdef.y}, $y_m$ is well-defined when $m$ is sufficiently large. Next
\begin{align*}
&\|S_{n}(\sigma^{m}\omega)y_{m}\varphi_{m}+ T_{n}(\sigma^{m}\omega)x_{m}\|_2^2 \\
= & \|S_{n}(\sigma^{m}\omega)y_{m}\|_2^2\varphi_{m}^2 + 2 \langle S_{n}(\sigma^{m}\omega)y_{m}, T_{n}(\sigma^{m}\omega)x_{m} \rangle \varphi_{m} + \|T_{n}(\sigma^{m}\omega)x_{m}\|_2^2,
\end{align*}
Thus
\begin{align*}
&\dfrac{\|S_{n}(\sigma^{m}\omega)y_{m}\varphi_{m}+ T_{n}(\sigma^{m}\omega)x_{m}\|_2^2}{\varphi_{m}^2\|S_{n}(\sigma^{m}\omega)y_{m}\|_2^2} \\
= & 1 + \dfrac{2 \langle S_{n}(\sigma^{m}\omega)y_{m}, T_{n}(\sigma^{m}\omega)x_{m} \rangle}{\varphi_{m}\|S_{n}(\sigma^{m}\omega)y_{m}\|_2^2} + \dfrac{\|T_{n}(\sigma^{m}\omega)x_{m}\|_2^2}{\varphi_{m}^2\|S_{n}(\sigma^{m}\omega)y_{m}\|_2^2}.
\end{align*}
Note
\[
\dfrac{|\langle S_{n}(\sigma^{m}\omega)y_{m}, T_{n}(\sigma^{m}\omega)x_{m}\rangle|}{\|S_{n}(\sigma^{m}\omega)y_{m}\|_2^2} \le \dfrac{\|T_{n}(\sigma^{m}\omega)x_{m}\|_2}{\|S_{n}(\sigma^{m}\omega)y_{m}\|_2}.
\]
Since $\varphi_{m}\to \infty$ as $m\to\infty$ almost surely, it is sufficient to prove that (for each $n$ that is sufficiently large)
\[
 \dfrac{\|T_{n}(\sigma^{m}\omega)x_{m}\|_2}{\|S_{n}(\sigma^{m}\omega)y_{m}\|_2} < C(n) < \infty \quad \text{ a.s.},
\]
where $C(n)$ is a bound depending only on $n$, not on $m$ or $\omega$. This follows by \autoref{lm:st.bnd}, that is, when $n$ is sufficiently large (not depending on $\omega, m, x_{m}$ or $y_{m}$),
\begin{align*}
\|T_{n}(\sigma^{m}\omega)x_{m}\|_2 & \le C_2(n),\\
\|S_{n}(\sigma^{m}\omega)y_{m}\|_2 & \ge C_0(n)>0.
\end{align*}
\end{proof}

\begin{thm} \label{thm:alignment}
Assume $\mu$ is strongly irreducible and contracting, and satisfies (FE1)-(FE3). Given a random vector $x(\omega)$ such that $\langle \widetilde{Z}(\omega),x(\omega) \rangle\not = 0$ and $\langle \widetilde{Z}(\sigma^{m}\omega), y_m\rangle\not = 0$ almost surely for all $m$ sufficiently large, if $\varphi_{n}(\omega, x(\omega))\to\infty \ \text{ a.s.}\ (n\to \infty)$, then
\[
\lim_{n\to\infty} \delta(S_{n}(\omega)x(\omega), \ T_{n}(\omega)x(\omega)) = 0\quad \text{ a.s.}.
\]
\end{thm}
\begin{proof}
By \autoref{lm:delta} and the fact that $\delta$ takes values in $[0,1]$, we have
\begin{align*}
\delta(x_{n+m}, y_{n+m}) \le & \dfrac{\|T_{n}(\sigma^m\omega)x_m\|_2}{\|S_{n}(\sigma^m\omega)y_m\varphi_m + T_{n}(\sigma^m\omega)x_m\|_2} \\
&+  \dfrac{\varphi_m\|S_{n}(\sigma^{m}\omega)y_m\|_2}{\|S_{n}(\sigma^m\omega)y_m\varphi_m + T_{n}(\sigma^m\omega)x_m\|_2}\dfrac{\|\wedge^2 S_{n}(\sigma^m\omega)\|_2}{\|S_{n}(\sigma^m\omega)x_m\|_2\|S_{n}(\sigma^{m}\omega)y_m\|_2}.
\end{align*}
Note $\langle \widetilde{Z}(\omega), x(\omega)\rangle \not= 0$ implies $\langle \widetilde{Z}(\sigma^{m}\omega), x_m\rangle\not = 0$ by \autoref{lm:zlm} \ref{zlm:tz}. And there exists an integer $M_1>0$ such that when $m>M_1$, $\langle \widetilde{Z}(\sigma^{m}\omega), y_m \rangle\not = 0\ \text{ a.s.}$. Thus we can apply \autoref{lm:wedges} to $x_m$ and $y_m$, and for $\epsilon>0$,
\begin{equation}\label{eq:term22}
\begin{aligned}
\mathbb{P}\left(\dfrac{\|\wedge^2 S_{n}(\sigma^m\omega)\|_2}{\|S_{n}(\sigma^m\omega)x_m\|_2\|S_{n}(\sigma^{m}\omega)y_m\|_2} \ge  \epsilon\quad \text{i.o. for }n\ge 1\right) =0
\end{aligned}
\end{equation}

By \autoref{lm:dom.varphi}, there exists an integer $N_1 > 0$ such that when $n > N_1$ we have
\[
\lim_{m\to\infty}\dfrac{\|S_{n}(\sigma^{m}\omega)y_{m}\varphi_{m}+ T_{n}(\sigma^{m}\omega)x_{m}\|_2}{\varphi_{m}\|S_{n}(\sigma^{m}\omega)y_{m}\|_2 } = 1\quad \text{ a.s.}.
\]
Thus for $\epsilon>0$, there exists an integer $M_2 := M_2(\epsilon, n)$ such that when $m > M_2$, we have
\begin{equation}\label{eq:term21}
1-\epsilon\le\dfrac{\|S_{n}(\sigma^{m}\omega)y_{m}\varphi_{m}+ T_{n}(\sigma^{m}\omega)x_{m}\|_2}{\varphi_m\|S_{n}(\sigma^{m}\omega)y_m\|} \le 1+\epsilon \quad \text{ a.s.}.
\end{equation}
Thus
\[
\dfrac{\|T_{n}(\sigma^m\omega)x_m\|_2}{\|S_{n}(\sigma^m\omega)y_m\varphi_m + T_{n}(\sigma^m\omega)x_m\|_2} \le \dfrac{\|T_{n}(\sigma^{m}\omega)x_m\|}{(1-\epsilon)\varphi_m\|S_{n}(\sigma^{m}\omega)x_m\|} \le \dfrac{C_2(n)}{(1-\epsilon)\varphi_mC_0(n)}.
\]
Since $\varphi_m\to\infty$ as $m\to\infty$, thus for any fixed $n > N_1$, we have
\begin{equation}\label{eq:term1}
\mathbb{P}\left(\dfrac{\|T_{n}(\sigma^m\omega)x_m\|_2}{\|S_{n}(\sigma^m\omega)y_m\varphi_m + T_{n}(\sigma^m\omega)x_m\|_2} \ge \epsilon \quad \text{i.o. for } m\ge 1 \right) = 0.
\end{equation}
Finally, by \eqref{eq:term1} \eqref{eq:term21} and \eqref{eq:term22},
\begin{align*}
& \mathbb{P}(\delta(x_n, y_n) \ge \epsilon\quad \text{i.o.})\\
 \le &  \mathbb{P}\left(\dfrac{\|T_{n}(\sigma^m\omega)x_m\|_2}{\|S_{n}(\sigma^m\omega)y_m\varphi_m + T_{n}(\sigma^m\omega)x_m\|_2} \ge \dfrac{\epsilon}{2} \quad \text{i.o. for } m\ge 1 \right)  \\
& + \mathbb{P}\left(\dfrac{\|\wedge^2 S_{n}(\sigma^{m}\omega)\|_2}{\|S_{n}(\sigma^{m}\omega)x_m\|_2\|S_{n}(\sigma^{m}\omega)y_m\|_2} \ge \dfrac{\epsilon(1-\epsilon)}{2} \quad \text{i.o. for }n\ge 1\right) \\
& = 0.
\end{align*}
This shows that $\delta(x_n, y_n)$ converges to 0 almost surely.
\end{proof}

\begin{cor}\label{cor:tz}
Assume $\mu$ is strongly irreducible and contracting, and satisfies (FE1)-(FE3). Assume $\mathbb{E}[\xi(\omega)]\not = 0$, then 
\begin{enumerate}[label = (\alph*)]
\item \label{tz:tz}
$T_{n}(\sigma^{-n}\omega)\cdot Z(\sigma^{-n}\omega)$ converges to $Z(\omega)$ almost surely as $n\to\infty$;
\item \label{tz:ttz}
$T_{n}^*(\omega)\cdot \widetilde{Z}(\sigma^{n}\omega)$ converges to $\widetilde{Z}(\omega)$ almost surely as $n\to\infty$;
\end{enumerate}
\end{cor}
\begin{proof}
We plan to apply $x(\omega) = Z(\omega)$ in \autoref{thm:alignment}. Note first by \autoref{prop:varphi.z}, when $\mathbb{E}[\xi(\omega)]\not= 0$, we have $\varphi_n(\omega, Z(\omega))\to\infty$ as $n\to\infty$. And $\langle \widetilde{Z}(\omega), Z(\omega) \rangle\not = 0\ \text{ a.s.}$ follows from \autoref{lm:zlm} \ref{zlm:nzero}. Next $\langle \widetilde{Z}(\sigma^{m}\omega), T_{m}(\omega)Z(\omega)\rangle\not = 0\ \text{ a.s.}$ follows from \eqref{eq:et.xi} and that $\mathbb{E}[\xi(\omega)]\not = 0$. Therefore, by \autoref{thm:alignment}, we have 
\[
\delta\big( S_{n}(\omega)\cdot Z(\omega),\ T_{n}(\omega)\cdot Z(\omega)\big) \to 0\quad \text{ a.s.} \ (n\to\infty).
\]
This proves \ref{tz:tz} since $S_{n}(\sigma^{-n}\omega)\cdot Z(\sigma^{-n}\omega) = Z(\omega)$ by \autoref{lm:zlm} \ref{zlm:z}. For the second statement, by the discussion at the end of \Cref{sec:varphi.z}, $\xi(\omega)$ shares the same law with $\xi^*(\omega)$ for the dual problem. Therefore $\mathbb{E}[\xi^*(\omega)] = \mathbb{E}[\xi(\omega)] \not = 0$. Consequently, we can similarly obtain that
\[
\delta\left(S_{n}^*(\omega)\cdot \widetilde{Z}(\sigma^{n}\omega), \ T_{n}^*(\omega)\cdot\widetilde{Z}(\sigma^{n}\omega)\right) \to 0 \quad \text{ a.s.} \ (n\to\infty).
\]
Then the rest follows from $S_{n}^*(\omega)\cdot \widetilde{Z}(\sigma^{n}\omega) = \widetilde{Z}(\omega)$.
\end{proof}

\section{Asymptotic Behaviour of $\varphi_n(\omega, x)$}
\label{sec:varphi.x}
In this section, we study the asymptotic behaviour of $\varphi_n(\omega, x)$ for a deterministic vector $x\not = 0$ by the help of $Z(\omega)$ using \autoref{prop:varphi.z} and \autoref{cor:tz}. The key idea is to use $\varphi_n(\omega, Z(\omega))$ to approximate $\varphi_n(\omega, x)$, and the vital step is to prove \autoref{prop:ts.finite}, which will be used to bound the difference between $\varphi_n(\omega, Z(\omega))$ and $\varphi_n(\omega, x)$ in \autoref{thm:varphi.x}. In the end, we will show that if $\mathbb{E}[\xi(\omega)]\not = 0$, $x_n$ and $y_n$ almost surely align in the same direction as $n\to\infty$ by \autoref{thm:alignment}. In this section, we specify the norms in the definition \eqref{defn:x} of $\varphi_n$ to be $\ell_2$-norms. Some of the results can be generalised by dropping the non-zero expectation condition, and this will be discussed in \Cref{sec:drop}.

We start by studying the action of $S_{n}(\omega)$ and $T_{n}(\omega)$ on $\widetilde{Z}(\omega)^\perp$. It will be useful for the bound in \autoref{prop:ts.finite}. The action of $S_{n}(\omega)$ on $\widetilde{Z}(\omega)$ is well-understood by Oseledets' multiplicative ergodic theorem or \cite[Corollary VI.1.7]{Bougerol}. Here we quote it as a lemma.

\begin{lemma}\label{lm:2le}
If $\mu$ is strongly irreducible and contracting, and satisfies (FE1) and (FE3), then for a non-zero vector $\eta$ which is orthogonal to $\widetilde{Z}(\omega)$, we have
\[
\limsup_{n\to\infty} \dfrac{1}{n}\log\|S_{n}(\omega)\eta\|\le\gamma_2<\gamma_1,
\]
where $\gamma_2$ is the second Lyapunov exponent.
\end{lemma}

Next, we will use this lemma to study the action of $T_{n}(\omega)$ on $\widetilde{Z}(\omega)^\perp$. Just as usual, we will use the inductive formula \eqref{eq:t.ind} and study the actions of $S_{n}(\sigma^{m}\omega)T_m(\omega)$ and $T_{n}(\sigma^{m}\omega)S_m(\omega)$ separately.

\begin{lemma}\label{lm:eta.bnd.st}
Let $\mu$ be strongly irreducible and contracting, and satisfies (FE1)-(FE3).
Assume $\eta\in \widetilde{Z}(\omega)^\perp$ with $\|\eta\|_2 = 1$ and $\mathbb{E}[\xi(\omega)]\not = 0$. For $\epsilon >0$, there exist integers $M, N>0$ such that when $m > M$ and $n > N$, we have
\[
\|S_m(\sigma^{n}\omega)T_{n}(\omega)\eta\|_2  \le \|S_{m}(\sigma^{n}\omega)\|_2\left(\|T_{n}^*(\omega)\widetilde{Z}(\sigma^{n}\omega)\|_2 \ \epsilon + \epsilon\right)\quad \text{ a.s.}.
\]
\end{lemma}
\begin{proof}
By \autoref{cor:tz}, for $\epsilon > 0$, there exists an integer $N>0$ such that when $n > N$,
\[
\delta\left(\dfrac{T_{n}^*(\omega)\widetilde{Z}(\sigma^{n}\omega)}{\|T_{n}^*(\omega)\widetilde{Z}(\sigma^{n}\omega)\|_2} ,\ \widetilde{Z}(\omega)\right) < \sqrt{2}\epsilon \quad \text{ a.s.}.
\]
Therefore by \autoref{lm:delta.equiv},
\begin{align*}
|\langle \widetilde{Z}(\sigma^{n}\omega), T_{n}(\omega)\eta\rangle| & =  |\langle T_{n}^*(\omega)\widetilde{Z}(\sigma^{n}\omega), \eta\rangle| \\
&\le \|T_{n}^*(\omega)\widetilde{Z}(\sigma^{n}\omega)\|_2 \left(|\langle \widetilde{Z}(\omega), \eta\rangle|+\epsilon\right) \\
& = \epsilon\|T_{n}^*(\omega)\widetilde{Z}(\sigma^{n}\omega)\|_2\quad \text{ a.s.}.
\end{align*}
Finally by \autoref{lm:ss}, 
\[
\lim_{m\to\infty}\dfrac{\|S_{m}(\sigma^{n}\omega)T_{n}(\omega)\eta\|_2}{\|S_{m}(\sigma^{n}\omega)\|_2} = |\langle \widetilde{Z}(\sigma^{n}\omega), T_{n}(\omega)\eta\rangle| \le \epsilon\|T_{n}^*(\omega)\widetilde{Z}(\sigma^{n}\omega)\|_2\quad \text{ a.s.}.
\]
The conclusion follows from the translation of the limit into the epsilon-delta language. Note since the convergence in \autoref{lm:ss} is uniform in $x$, the choice of $M$ does not depend on $n$.
\end{proof}

\begin{prop}
\label{prop:ts.finite}
Let $\mu$ be strongly irreducible and contracting, and satisfy (FE1)-(FE3). If $\mathbb{E}[\xi(\omega)] \not = 0$, then
\[
\sup_n\dfrac{1}{n}\dfrac{\|T_{n}(\omega)\|_2}{\|S_{n}(\omega)\|_2} < \infty\quad \text{ a.s.}.
\]
\end{prop}
\begin{proof}
Assume $\|T_{n}(\omega)\|_2 = \|T_{n}(\omega)u_n(\omega)\|_2$ for some vector $\|u_n(\omega)\| = 1$. Then we can decompose $u_n(\omega) = \alpha_n(\omega) z(\omega) + \beta_n(\omega)\eta$, where $z(\omega)$ is a representative of $Z(\omega)$ and $\eta\in \widetilde{Z}(\omega)^\perp$ with $\|\eta\|_2=1$ and
\[
\alpha_n(\omega)=\dfrac{\langle \widetilde{Z}(\omega), u_n(\omega) \rangle}{\langle \widetilde{Z}(\omega), z(\omega)\rangle}, \quad |\beta_n(\omega)|=|u_n(\omega)-\alpha_n(\omega)z(\omega)|.
\]
By \autoref{lm:zlm} \ref{zlm:nzero}, $\alpha_n(\omega)$ is well-defined almost surely. Since
\[
\|T_{n}(\omega)\|_2 \le |\alpha_n(\omega)|\|T_{n}(\omega)Z(\omega)\|_2 + |\beta_n(\omega)|\|T_{n}(\omega)\eta\|_2.
\]
and that
\[
|\alpha_n(\omega)| \le \dfrac{1}{|\langle \widetilde{Z}(\omega), Z(\omega) \rangle|}, \quad
|\beta_n(\omega)| \le 1 + \dfrac{1}{|\langle \widetilde{Z}(\omega), Z(\omega) \rangle|}\quad \text{ a.s.},
\]
the problem reduces to proving
\begin{enumerate}[label = (\alph*)]
\item \label{ts.finite.z}
$\sup_n\dfrac{\|T_{n}(\omega)Z(\omega)\|_2}{n\|S_{n}(\omega)\|_2} < \infty\quad \text{ a.s.}.$
\item \label{ts.finite.eta}
$\sup_n\dfrac{\|T_{n}(\omega)\eta\|_2}{n\|S_{n}(\omega)\|_2} < \infty\quad \text{ a.s.}.$
\end{enumerate}
For \ref{ts.finite.z}, by \autoref{prop:varphi.z} and \autoref{lm:ss} we have
\[
\lim_{n\to\infty} \dfrac{\|T_{n}(\omega)Z(\omega)\|_2}{n\|S_{n}(\omega)\|_2} 
= \lim_{n\to\infty} \dfrac{\|T_{n}(\omega)Z(\omega)\|_2}{n\|S_{n}(\omega)Z(\omega)\|_2} \dfrac{\|S_{n}(\omega)Z(\omega)\|_2}{\|S_{n}(\omega)\|_2}
= |\mathbb{E}[\xi(\omega)]|\cdot|\langle \widetilde{Z}(\omega), Z(\omega)\rangle|\quad \text{ a.s.}.
\]
Thus when $n$ is sufficiently large, $\|T_{n}(\omega)Z(\omega)\|_2/n\|S_{n}(\omega)\|_2$ is universally bounded by the limit plus $\epsilon$ for some $\epsilon>0$.

For \ref{ts.finite.eta}, note first by \autoref{lm:eta.bnd.st}, for $\epsilon > 0$, we can find integers $M_1, N_1>0$ such that when $m > M_1$ and $n > N_1$,
\[
\|S_m(\sigma^{n}\omega)T_{n}(\omega)\eta\|_2  \le \|S_{m}(\sigma^{n}\omega)\|_2\left(\|T_{n}^*(\omega)\widetilde{Z}(\sigma^{n}\omega)\|_2 \ \epsilon + \epsilon\right)\quad \text{ a.s.}.
\]
This shows
\begin{equation}\label{eq:st.eta}
\dfrac{\|S_m(\sigma^{n}\omega)T_{n}(\omega)\eta\|_2}{(m+n)\|S_{m+n}(\omega)\|_2} \le
   \dfrac{\|S_{m}(\sigma^{n}\omega)\|_2\|S_{n}(\omega)\|_2}{\|S_{m+n}(\omega)\|_2}\dfrac{\|T_{n}^*(\omega)\widetilde{Z}(\sigma^{n}\omega)\|_2 \ \epsilon + \epsilon}{n\|S_{n}(\omega)\|_2}\quad \text{ a.s.}.
\end{equation}
By \autoref{cor:varphi.tz}, there exists an integer $N_2>0$ such that when $n > N_2$, we have 
\begin{equation}\label{eq:st.eta2}
\dfrac{\|T_{n}^*(\omega)\widetilde{Z}(\sigma^{n}\omega)\|_2 \ \epsilon}{n\|S_{n}(\omega)\|_2}
\le \dfrac{\|T_{n}^*(\omega)\widetilde{Z}(\sigma^{n}\omega)\|_2 \ \epsilon}{n\|S_{n}^*(\omega)\widetilde{Z}(\sigma^{n}\omega)\|_2} < |\mathbb{E}[\xi(\omega)]|\epsilon + \epsilon^2\quad \text{ a.s.}.
\end{equation}
Moreover, by \autoref{cor:ss}, for a fixed $n\ge 1$ we have
\begin{equation}\label{eq:st.eta1}
\sup_m \dfrac{\|S_{m}(\sigma^{n}\omega)\|_2}{\|S_{m+n}(\omega)\|_2} < \infty\quad \text{ a.s.}.
\end{equation}
Therefore, by fixing $n > \max\{N_1, N_2\}$ and applying \eqref{eq:st.eta1} and \eqref{eq:st.eta2} to \eqref{eq:st.eta}, we have for all $m >  M_1$,
\begin{equation} \label{eq:st.eta.fin}
\begin{aligned}
&\dfrac{\|S_m(\sigma^{n}\omega)T_{n}(\omega)\eta\|_2}{(m+n)\|S_{m+n}(\omega)\|_2} \\
\le & \sup_m\dfrac{\|S_{m}(\sigma^{n}\omega)\|_2\|S_{n}(\omega)\|_2}{\|S_{m+n}(\omega)\|_2} \left(|\mathbb{E}[\xi(\omega)]|\epsilon+\epsilon^2\right) + \sup_m\dfrac{\|S_{m}(\sigma^{n}\omega)\|_2}{n\|S_{m+n}(\omega)\|_2}\epsilon \\
< & \infty. \quad \text{ a.s.}. 
\end{aligned}
\end{equation}

On the other hand, by \autoref{lm:2le}, for $\epsilon>0$, there exists an integer $N_3>0$ such that when $n > N_3$, we have
\[
\|S_{n}(\omega)\eta\|_2 \le e^{n(\gamma_2+\epsilon)},\quad \|S_{m+n}(\omega)\|_2\ge e^{(m+n)(\gamma_1-\epsilon)}\quad  \text{ a.s.}.
\]
Choose an integer $M_2>0$ such that when $m > M_2$, $\|T_{m}(\sigma^{n}\omega)\|\le C_2(m)$ by \autoref{lm:st.bnd}.
Then by fixing $m>\{M_1, M_2\}$ and choosing $\epsilon>0$ with $2\epsilon < \gamma_1-\gamma_2$, we have that for any $n > \max\{N_1, N_3\}$,
\begin{equation}\label{eq:ts.eta.fin}
\dfrac{\|T_{m}(\sigma^{n}\omega)S_{n}(\omega)\eta\|_2}{\|S_{m+n}(\omega)\|_2} \le \|T_{m}(\sigma^{n}\omega)\|_2 \ e^{n(\gamma_2-\gamma_1+2\epsilon)}e^{-m\gamma_1+m\epsilon}\le C_2(m) e^{-m\gamma_1+m\epsilon}\quad \text{ a.s.}.
\end{equation}
Then \ref{ts.finite.eta} follows from \eqref{eq:st.eta.fin} and \eqref{eq:ts.eta.fin}. This finishes the proof of the proposition.
\end{proof}

\begin{thm}\label{thm:varphi.x}
Let $\mu$ be strongly irreducible and contracting, and satisfy (FE2) and (FE4). If $\mathbb{E}[\xi(\omega)]\not = 0$, then for a non-zero deterministic vector $x\not = 0$,
\[
\lim_{n\to\infty}\dfrac{\varphi_n(\omega, x)}{n} = |\mathbb{E}[\xi(\omega)]| \quad \text{ a.s.}.
\]
\end{thm}
\begin{proof}
By \eqref{eq:t.ind}, we have
\[
T_{m+n}(\sigma^{-n}\omega)x = S_{m}(\omega)T_{n}(\sigma^{-n}\omega)x + T_{m}(\omega)S_{n}(\sigma^{-n}\omega)x.
\]
Since $S_{n}(\sigma^{-n}\omega)\cdot[x]$ converges to $Z(\omega)$ almost surely under (FE4) as $n\to\infty$ (\cite[Theorem 6.3.1]{Bougerol}), for $\epsilon>0$, there exists an integer $N>0$ such that when $n > N$, we have
\[
\delta\left(\dfrac{S_{n}(\sigma^{-n}\omega)x}{\|S_{n}(\sigma^{-n}\omega)x\|_2}, \ Z(\omega)\right)<\sqrt{2}\epsilon\quad \text{ a.s.}.
\]
Then by \autoref{lm:delta.equiv},
\begin{align*}
\left|\dfrac{\|T_{m}(\omega)S_{n}(\sigma^{-n}\omega)x\|_2}{\|S_{n}(\sigma^{-n}\omega)x\|_2} - \|T_{m}(\omega)Z(\omega)\|_2 \right| 
&\le\dfrac{\|T_{m}(\omega)\|_2}{\sqrt{2}}\cdot\delta\left(\dfrac{S_{n}(\sigma^{-n}\omega)x}{\|S_{n}(\sigma^{-n}\omega)x\|_2}, \ Z(\omega)\right)\\
& < \epsilon\|T_{m}(\omega)\|_2\quad \text{ a.s.},
\end{align*}
and similarly,
\begin{align*}
\left|\dfrac{\|S_{m}(\omega)S_{n}(\sigma^{-n}\omega)x\|_2}{\|S_{n}(\sigma^{-n}\omega)x\|_2} - \|S_{m}(\omega)Z(\omega)\|_2 \right| 
 < \epsilon\|S_{m}(\omega)\|_2\quad \text{ a.s.}.
\end{align*}
This shows
\begin{equation}\label{eq:varphi.x.ts}
\dfrac{\|T_{m}(\omega)Z(\omega)\|_2 - \epsilon\|T_{m}(\omega)\|_2}{\|S_{m}(\omega)Z(\omega)\|_2+\epsilon\|S_{m}(\omega)\|_2}
\le \dfrac{\|T_{m}(\omega)S_{n}(\sigma^{-n}\omega)x\|_2}{\|S_{m}(\omega)S_{n}(\sigma^{-n}\omega)x\|_2}
\le \dfrac{\|T_{m}(\omega)Z(\omega)\|_2 + \epsilon\|T_{m}(\omega)\|_2}{\|S_{m}(\omega)Z(\omega)\|_2-\epsilon\|S_{m}(\omega)\|_2}\quad \text{ a.s.}.
\end{equation}
Starting with looking at the right inequality (for convenience, denote by RHS the rightmost term of \eqref{eq:varphi.x.ts}),
\begin{align*}
&\text{RHS} - \varphi_m(\omega, Z(\omega))\\
= & \dfrac{\|T_{m}(\omega)Z(\omega)\|_2 + \epsilon\|T_{m}(\omega)\|_2}{\|S_{m}(\omega)Z(\omega)\|_2-\epsilon\|S_{m}(\omega)\|_2} - \dfrac{\|T_{m}(\omega)Z(\omega)\|_2}{\|S_{m}(\omega)Z(\omega)\|_2} \\
=& \dfrac{\|T_{m}(\omega)\|_2\|S_{m}(\omega)Z(\omega)\|_2 + \|S_{m}(\omega)\|_2\|T_{m}(\omega)Z(\omega)\|_2}{\|S_{m}(\omega)Z(\omega)\|_2(\|S_{m}(\omega)Z(\omega)\|_2 - \epsilon\|S_{m}(\omega)\|_2)}\epsilon \\
=& \left(\dfrac{\|T_{m}(\omega)\|_2/\|S_{m}(\omega)\|_2}{\|S_{m}(\omega)Z(\omega)\|_2/\|S_{m}(\omega)\|_2 -\epsilon} + \dfrac{\varphi_m(\omega, Z(\omega))
}{\|S_{m}(\omega)Z(\omega)\|_2/\|S_{m}(\omega)\|_2 - \epsilon}\right)\epsilon
\end{align*}
By \autoref{prop:ts.finite}, $\|T_{m}(\omega)\|_2/m\|S_{m}(\omega)\|_2$ has a universal upper bound almost surely for $m\ge1$. By \autoref{lm:ss}, $\|S_{m}(\omega)Z(\omega)\|_2/\|S_{m}(\omega)\|_2$ has a universal lower bound almost surely for $m\ge1$. By \autoref{prop:varphi.z}, $\varphi_m(\omega, Z(\omega))/m$ has a universal upper bound almost surely for $m\ge 1$. Therefore, there exists $\Gamma(\omega)$ such that
\begin{equation}\label{eq:varphi.x.rhs}
\dfrac{\text{RHS}}{m} \le \dfrac{\varphi_m(\omega, Z(\omega))}{m} + \Gamma(\omega)\epsilon\quad \text{ a.s.}.
\end{equation}
Similarly using LHS to denote the leftmost term in \eqref{eq:varphi.x.ts}, one can similarly obtain $\Gamma'(\omega)$ such that
\begin{equation}\label{eq:varphi.x.lhs}
\dfrac{\varphi_m(\omega, Z(\omega))}{m} - \Gamma'(\omega)\epsilon \le \dfrac{\text{LHS}}{m}\quad \text{ a.s.}.
\end{equation}
Now
\begin{equation}\label{eq:varphi.x.mpn}
\dfrac{\varphi_{m+n}(\sigma^{-n}\omega, x)}{m+n} \ge \dfrac{1}{m+n}\left(\dfrac{\|T_{m}(\omega)S_{n}(\sigma^{-n}\omega)x\|_2}{\|S_{m}(\omega)S_{n}(\sigma^{-n}\omega)x\|_2} - \dfrac{\|S_{m}(\omega)T_{n}(\sigma^{-n}\omega)x\|_2}{\|S_{m}(\omega)S_{n}(\sigma^{-n}\omega)x\|_2}\right)\quad \text{ a.s.}.
\end{equation}
By \autoref{lm:ss}, there exists an integer $M>0$ (not depending on $n$) such that when $m>M$, 
\begin{equation}\label{eq:varphi.x.term2}
 \dfrac{\|S_{m}(\omega)T_{n}(\sigma^{-n}\omega)x\|_2}{\|S_{m}(\omega)S_{n}(\sigma^{-n}\omega)x\|_2} \le \dfrac{|\langle \widetilde{Z}(\omega), T_{n}(\sigma^{-n}\omega)x\rangle|}{|\langle \widetilde{Z}(\omega), S_{n}(\sigma^{-n}\omega)x \rangle|} + \epsilon\quad \text{ a.s.}.
\end{equation}
Applying \eqref{eq:varphi.x.term2} and \eqref{eq:varphi.x.ts} to \eqref{eq:varphi.x.mpn} and then using \eqref{eq:varphi.x.lhs} gives
\begin{equation}\label{eq:varphi.x.liminf}
\begin{aligned}
\dfrac{\varphi_{m+n}(\sigma^{-n}\omega, x)}{m+n} &\ge \dfrac{1}{m+n}\left(\text{LHS} - \dfrac{|\langle \widetilde{Z}(\omega), T_{n}(\sigma^{-n}\omega)x\rangle|}{|\langle \widetilde{Z}(\omega), S_{n}(\sigma^{-n}\omega)x \rangle|} - \epsilon \right) \\
& \ge \dfrac{m}{m+n}\left(\dfrac{\varphi_m(\omega, Z(\omega))}{m} - \Gamma'(\omega)\epsilon\right) - \dfrac{1}{m+n}\left(\dfrac{|\langle \widetilde{Z}(\omega), T_{n}(\sigma^{-n}\omega)x\rangle|}{|\langle \widetilde{Z}(\omega), S_{n}(\sigma^{-n}\omega)x \rangle|} + \epsilon\right)\quad \text{ a.s.}.
\end{aligned}
\end{equation}
By fixing an $n > N$, letting $m\to\infty$ in \eqref{eq:varphi.x.liminf} and applying \autoref{prop:varphi.z}, we have
\begin{equation}\label{eq:varphi.x.inf}
\liminf_{m\to\infty}\dfrac{\varphi_{m}(\sigma^{-n}\omega, x)}{m} \ge |\mathbb{E}[\xi(\omega)]| - \Gamma'(\omega)\epsilon \quad \text{ a.s.}.
\end{equation}
Similar to \eqref{eq:varphi.x.mpn}, we obtain the upper bound
\begin{equation}\nonumber
\begin{aligned}
\dfrac{\varphi_{m+n}(\sigma^{-n}\omega, x)}{m+n} 
&\le \dfrac{1}{m+n}\left(\dfrac{\|T_{m}(\omega)S_{n}(\sigma^{-n}\omega)x\|_2}{\|S_{m}(\omega)S_{n}(\sigma^{-n}\omega)x\|_2} + \dfrac{\|S_{m}(\omega)T_{n}(\sigma^{-n}\omega)x\|_2}{\|S_{m}(\omega)S_{n}(\sigma^{-n}\omega)x\|_2}\right)\quad \\
& \le \dfrac{1}{m+n}\left(\text{RHS} + \dfrac{|\langle \widetilde{Z}(\omega), T_{n}(\sigma^{-n}\omega)x\rangle|}{|\langle \widetilde{Z}(\omega), S_{n}(\sigma^{-n}\omega)x \rangle|} + \epsilon\right)  \\
& \le \dfrac{m}{m+n}\left(\dfrac{\varphi_m(\omega, Z(\omega))}{m} + \Gamma(\omega)\epsilon\right) + \dfrac{1}{m+n}\left(\dfrac{|\langle \widetilde{Z}(\omega), T_{n}(\sigma^{-n}\omega)x\rangle|}{|\langle \widetilde{Z}(\omega), S_{n}(\sigma^{-n}\omega)x \rangle|} + \epsilon\right).
\end{aligned}
\end{equation}
Similarly by fixing an $n > N$, letting $m\to\infty$ and applying \autoref{prop:varphi.z}, we have
\begin{equation}\label{eq:varphi.x.sup}
\limsup_{m\to\infty}\dfrac{\varphi_{m}(\sigma^{-n}\omega, x)}{m} \le |\mathbb{E}[\xi(\omega)]| + \Gamma(\omega)\epsilon \quad \text{ a.s.}.
\end{equation}
Now since $\epsilon>0$ is arbitrary, \eqref{eq:varphi.x.inf} and \eqref{eq:varphi.x.sup} together give
\[
\lim_{m\to\infty}\dfrac{\varphi_{m}(\sigma^{-n}\omega, x)}{m} = |\mathbb{E}[\xi(\omega)]| \quad \text{ a.s.}.
\]
Since $\sigma$ is measure-preserving, the conclusion follows.
\end{proof}

\begin{cor}\label{cor:alignment.x}
Let $\mu$ be strongly irreducible and contracting, and satisfy (FE2) and (FE4). If $\mathbb{E[\xi(\omega)]}\not= 0$, then for a given non-zero deterministic vector $x\not =0$,
\[
\lim_{n\to\infty}\delta\left(S_{n}(\omega)x,\  T_{n}(\omega)x\right) = 0.
\]
\end{cor}
\begin{proof}
We plan to let $x(\omega)$ be a deterministic vector $x\not = 0$ in \autoref{thm:alignment}.
Notice first by \autoref{thm:varphi.x}, when $\mathbb{E}[\xi(\omega)]\not = 0$, we have $\varphi_n(n,x)\to\infty\ \text{ a.s.}$. And 
\begin{equation}\label{eq:det.x.nzero}
\mathbb{P}(\langle \widetilde{Z}(\omega), x\rangle=0) = \int \delta_{\widetilde{Z}(\omega)}(x^\perp)\mathbb{P}(\operatorname{d}\omega) = \nu^*(x^\perp) = 0,
\end{equation}
where the second equality follows from \autoref{lm:zlm} \ref{zlm:nust}, and the last follows from the fact \cite[Proposition III.2.3]{Bougerol} that when $\mu^*$ is strongly irreducible, $\nu^*$ is proper (that is, taking value $0$ on any hyperplane). Next we want to check that $\langle \widetilde{Z}(\sigma^{m}\omega), y_m\rangle\not = 0\ \text{ a.s.}$ for all $m$ sufficiently large, which is equivalent to $\langle \widetilde{Z}(\sigma^{m}\omega), T_{m}(\omega)x\rangle\not= 0$, since $T_{m}(\omega)x\not = 0$ for $m$ sufficiently large by \autoref{lm:wdef.y}. Now by \autoref{cor:tz}, for $\epsilon>0$, when $m$ is sufficiently large, we have
\[
\delta( T_{n}^*(\omega)\widetilde{Z}(\sigma^{m}\omega), \widetilde{Z}(\omega)) < \sqrt{2}\epsilon
\]
Thus by \autoref{lm:delta.equiv}
\[
|\langle \widetilde{Z}(\sigma^{m}\omega), T_{m}(\omega)x\rangle| = |\langle T_{m}^*(\omega)\widetilde{Z}(\sigma^{m}\omega), x\rangle| \ge |\langle \widetilde{Z}(\omega), x\rangle| - \epsilon\|x\|.
\]
By \eqref{eq:det.x.nzero}, $|\langle \widetilde{Z}(\omega), x\rangle|>0$ almost surely. By choosing $\epsilon>0$ sufficiently small, we can obtain $|\langle \widetilde{Z}(\sigma^{m}\omega), T_{m}(\omega)x\rangle|\not = 0\ \text{ a.s.}$. Then the conclusion follows from \autoref{thm:alignment}.
\end{proof}

\begin{cor} \label{cor:t.asymp}
Let $\mu$ be strongly irreducible and contracting, and satisfy (FE2) and (FE4). Assume $\mathbb{E[\xi(\omega)]}\not= 0$. Then
\begin{enumerate}[label=(\alph*)]
\item \label{t:range}
a limit point of $T_{n}^*(\omega)/\|T_{n}^*(\omega)\|$ is almost surely a rank one matrix, with range being the one-dimensional vector space spanned by $\widetilde{Z}(\omega)$;
\item \label{t:tz}
for a random vector $x(\omega)$, we have 
\[
\lim_{n\to\infty} \dfrac{\|T_{n}(\omega)x(\omega)\|_2}{\|T_{n}(\omega)\|_2} = |\langle \widetilde{Z}(\omega), x(\omega)\rangle|\quad \text{ a.s.};
\]
\item \label{t:ts}
\[
\lim_{n\to\infty}\dfrac{1}{n}\dfrac{\|T_{n}(\omega)\|_2}{\|S_{n}(\omega)\|_2} = |\mathbb{E}[\xi(\omega)]|\quad \text{ a.s.}.
\]
\item \label{t:norm.t}
\[
\lim_{n\to\infty}\dfrac{1}{n}\log\|T_{n}(\omega)\| = \gamma \quad \text{ a.s.}.
\]
\end{enumerate}
\end{cor}
\begin{proof}
To prove \ref{t:range}, note for the standard basis $e_1, \dots, e_d$, we have
\[
\lim_{n\to\infty} \delta(T_{n}^*(\omega)e_i,\ S_{n}^*(\omega)e_i) = 0 \ (1\le i\le d) \quad \text{ a.s.}
\]
As $S_{n}^*(\omega)\cdot [e_i]$ converges in direction to $\widetilde{Z}(\omega)$ almost surely, so does $T_{n}^*(\omega)\cdot [e_i]$. Therefore $T_{n}^*(\omega)$ has the one-dimensional range spanned by $\widetilde{Z}(\omega)$ after normalisation. Next, since any limit point of $T_{n}^*(\omega)/\|T_{n}^*(\omega)\|$ is almost surely of rank one, \ref{t:tz} follows by the same proof as that of \autoref{lm:ss}. For \ref{t:ts}, we have
\begin{align*}
\lim_{n\to\infty}\dfrac{\|T_{n}(\omega)\|_2}{n\|S_{n}(\omega)\|_2} 
= \lim_{n\to\infty} \dfrac{\|T_{n}(\omega)\|_2}{\|T_{n}(\omega)Z(\omega)\|_2}\dfrac{\|T_{n}(\omega)Z(\omega)\|_2}{n\|S_{n}(\omega)Z(\omega)\|_2}\dfrac{\|S_{n}(\omega)Z(\omega)\|_2}{\|S_{n}(\omega)\|_2} = |\mathbb{E}[\xi(\omega)]|\quad \text{ a.s.},
\end{align*}
by \ref{t:tz}, \autoref{prop:varphi.z}, and \autoref{lm:ss} respectively. \ref{t:norm.t} follows immediately from \ref{t:ts}. And we have freedom to choose a different matrix norm in \ref{t:norm.t} because they only differ by a scalar constant, which does not contribute to the limit.
\end{proof}

\section{Asymptotic Behaviours of $\psi_n(\omega,x)$}
\label{sec:psi.x}

This section studies the asymptotic behaviour of $\psi_n(\omega, x)$ given a deterministic vector $x\not = 0$ using the results from the previous section. As we have proved in \autoref{cor:alignment.x} that when $\mathbb{E}[\xi(\omega)]\not =0$, $S_{n}(\omega)x$ and $T_{n}(\omega)x$ almost surely align to the same direction as $n\to\infty$, the limit of $|\psi_n(\omega, x)|/n$ is immediate. The essential part of the proof is to discuss the sign by comparing $\psi_n(\omega, x)$ with $\psi_n(\omega, Z(\omega))$. Moreover, \autoref{cor:alignment.x} only gives the alignment of $S_{n}(\omega)x$ and $T_{n}(\omega)x$ as equivalent classes in the projective space, that is, they might differ by a sign as normalised vectors. As a consequence of \autoref{thm:psi}, we will determine the sign depending on the sign of $\mathbb{E}[\xi(\omega)]$. The assumption of $\mathbb{E}[\xi(\omega)]\not= 0$ in \autoref{thm:psi} can be dropped with a linear trick, which will be discussed in \Cref{sec:drop}.

\begin{thm}\label{thm:psi}
Let $\mu$ be strongly irreducible and contracting, and satisfy (FE2) and (FE4). Assume $\mathbb{E}[\xi(\omega)]\not = 0$, then for a deterministic vector $x\not = 0$,
\[
\lim_{n\to\infty}\dfrac{\psi_n(\omega, x)}{n} = \mathbb{E}[\xi(\omega)]\quad \text{ a.s.}.
\]
\end{thm}
\begin{proof}
By \autoref{ind.lm:psi},
\[
\psi_{n+m}(\omega, x) = \dfrac{\langle S_{n}(\sigma^{m}\omega)x_m, S_{n}(\sigma^{m}\omega)y_m \rangle}{\|S_{n}(\omega)x_m\|^2_2}\varphi_m(\omega, x) + \psi_n(\omega, x_m).
\]
First notice that by \autoref{lm:st.bnd}, there exists an integer $N_1 > 0$ such that for $n > N_1$ we have 
\[
|\psi_n(\omega, x_m)| \le \dfrac{\|T_{n}(\omega)x_m\|_2}{\|S_{n}(\omega)x_m\|_2} \le \dfrac{C_2(n)}{C_0(n)} \quad \text{ a.s.}.
\]
Next by the uniform convergence in \autoref{cor:sss} and 
\[
\mathbb{P}(\langle \widetilde{Z}(\sigma^{m}\omega), x_m\rangle=0) = \mathbb{P}(\langle \widetilde{Z}(\omega), x\rangle = 0) = \nu^*(x^\perp) = 0,
\]
we have that for $\epsilon>0$, there exists an integer $N_2>0$ (not depending on $m$) such that when $n > N_2$ we have
\[
\dfrac{\langle \widetilde{Z}(\sigma^{m}\omega), y_m \rangle}{\langle \widetilde{Z}(\sigma^{m}\omega), x_m\rangle} - \epsilon 
\le \dfrac{\langle S_{n}(\sigma^{m}\omega)x_m, S_{n}(\sigma^{m}\omega)y_m \rangle}{\|S_{n}(\omega)x_m\|^2_2} 
\le \dfrac{\langle \widetilde{Z}(\sigma^{m}\omega), y_m \rangle}{\langle \widetilde{Z}(\sigma^{m}\omega), x_m\rangle} + \epsilon\quad \text{ a.s.}.
\]
That is, for a fixed $n > \max\{N_1, N_2\}$, we have
\[
 \left(\dfrac{\langle \widetilde{Z}(\sigma^{m}\omega), y_m \rangle}{\langle \widetilde{Z}(\sigma^{m}\omega), x_m\rangle}-\epsilon\right)\dfrac{\varphi_m}{m} - \dfrac{1}{m}\dfrac{C_2(n)}{C_0(n)}
\le \dfrac{\psi_{n+m}(\omega, x)}{m} \le \left(\dfrac{\langle \widetilde{Z}(\sigma^{m}\omega), y_m \rangle}{\langle \widetilde{Z}(\sigma^{m}\omega), x_m\rangle}+\epsilon\right)\dfrac{\varphi_m}{m} + \dfrac{1}{m}\dfrac{C_2(n)}{C_0(n)}\quad \text{ a.s.}.
\]
Since $\varphi_m/m\to|\mathbb{E}[\xi(\omega)]|$ by \autoref{thm:varphi.x}, it suffices to prove that
\begin{equation}\label{eq:psi.goal}
\lim_{m\to\infty} \dfrac{\langle \widetilde{Z}(\sigma^{m}\omega), y_m \rangle}{\langle \widetilde{Z}(\sigma^{m}\omega), x_m\rangle} =
\begin{cases}
1 & \mathbb{E}[\xi(\omega)] > 0, \\
-1 & \mathbb{E}[\xi(\omega)] < 0,
\end{cases}
\end{equation}
then the conclusion would follow by taking $m\to\infty$ in the previous inequality.

By \eqref{eq:et.xi}, if $\mathbb{E}[\xi(\omega)]>0$, then for sufficiently large $m$,
\begin{equation}\label{eq:psi.pos}
\dfrac{1}{m}\dfrac{\langle \widetilde{Z}(\sigma^{m}\omega), T_{m}(\omega)Z(\omega) \rangle}{\langle \widetilde{Z}(\sigma^{m}\omega), S_{m}(\omega)Z(\omega)\rangle} 
=  \dfrac{1}{m}\dfrac{\langle T_{m}^*(\omega)\widetilde{Z}(\sigma^{m}\omega), Z(\omega) \rangle}{\langle S_{m}^*(\omega)\widetilde{Z}(\sigma^{m}\omega), Z(\omega)\rangle} 
> 0 \quad \text{a.s.}.
\end{equation}
By fixing a normalised vector $\widetilde{z}(\omega)$ as a representative of $\widetilde{Z}(\omega)\in \mathcal{P}(\mathbb{R}^d)$, choosing 
\[
\widetilde{z}(\sigma^{m}\omega):= \overline{(S_{m}^*(\omega))^{-1}\widetilde{z}(\omega)}
\] 
as a representative of $\widetilde{Z}(\sigma^{m}\omega)$, then by \autoref{cor:tz} we can write
\begin{equation}\label{eq:psi.decomp}
\dfrac{T_{m}^*(\omega)\widetilde{z}(\sigma^{m}\omega)}{\|T_{m}^*(\omega)\widetilde{z}(\sigma^{m}\omega)\|_2} = \chi_m(\omega)\widetilde{z}(\omega) + \zeta_m(\omega),
\end{equation}
where $\chi_m(\omega)$ takes values in $\{\pm 1\}$ and $\|\zeta_m(\omega)\|_2\to 0$ as $m\to\infty$.
Then
\begin{equation}\label{eq:psi.chi.zeta}
\begin{aligned}
\dfrac{1}{m}\dfrac{\langle T_{m}^*(\omega)\widetilde{z}(\sigma^{m}\omega), Z(\omega) \rangle}{\langle S_{m}^*(\omega)\widetilde{z}(\sigma^{m}\omega), Z(\omega)\rangle}
&= \dfrac{1}{m}\dfrac{\|T_{m}^*(\omega)\widetilde{z}(\sigma^m\omega)\|_2}{\|S_{m}^*(\omega)\widetilde{z}(\sigma^m\omega)\|_2}\dfrac{\langle \overline{T_{m}^*(\omega)\widetilde{z}(\sigma^m\omega)}, Z(\omega)\rangle}{\langle \widetilde{z}(\omega),Z(\omega) \rangle} \\
& \le \dfrac{1}{m}\dfrac{\|T_{m}^*(\omega)\widetilde{z}(\sigma^m\omega)\|_2}{\|S_{m}^*(\omega)\widetilde{z}(\sigma^{m}\omega)\|_2}\left(\chi_m(\omega) + \dfrac{\langle \zeta_m(\omega), Z(\omega) \rangle}{\langle \widetilde{z}(\omega), Z(\omega)\rangle}\right)\quad \text{ a.s.}.
\end{aligned}
\end{equation}
By \autoref{lm:zlm} \ref{zlm:fin.inv}, $1/ |\langle \widetilde{z}(\omega), Z(\omega)\rangle|$ is uniformly bounded for almost all $\omega$, thus
\begin{equation}\label{eq:zeta.small}
\left|\dfrac{\langle \zeta_m(\omega), Z(\omega) \rangle}{\langle \widetilde{z}(\omega), Z(\omega)\rangle}\right| \le \dfrac{\|\zeta_m(\omega)\|_2}{|\langle \widetilde{z}(\omega),Z(\omega) \rangle|} \to 0 \quad \text{ a.s.},
\end{equation}
as $m\to\infty$. Thus \eqref{eq:psi.pos} \eqref{eq:psi.chi.zeta} and \autoref{cor:varphi.tz} show that $\chi_m(\omega) > 0$ when $m$ is sufficiently large. That is, $\chi_m(\omega) = 1\ \text{ a.s.}$, when $m$ is sufficiently large. 

Now we are ready to prove \eqref{eq:psi.goal} when $\mathbb{E}[\xi(\omega)]>0$. By \autoref{cor:alignment.x} we have 
\[
\lim_{m\to\infty}\delta(x_m, y_m) = 0,
\]
thus
\[
\lim_{m\to\infty} \left|\dfrac{\langle \widetilde{Z}(\sigma^{m}\omega), y_m \rangle}{\langle \widetilde{Z}(\sigma^{m}\omega), x_m\rangle}\right| = 1.
\]
It suffices to prove that, for sufficiently large $m$, we have
\[
\dfrac{\langle \widetilde{Z}(\sigma^{m}\omega), T_{m}(\omega)x \rangle}{\langle \widetilde{Z}(\sigma^{m}\omega), S_{m}(\omega)x\rangle} > 0 \quad \text{ a.s.}.
\]
Now by \eqref{eq:psi.decomp},
\begin{align*}
\dfrac{\langle \widetilde{Z}(\sigma^{m}\omega), T_{m}(\omega)x \rangle}{\langle \widetilde{Z}(\sigma^{m}\omega), S_{m}(\omega)x\rangle} 
&= \dfrac{\langle T_{m}^*(\omega)\widetilde{z}(\sigma^{m}\omega), x \rangle}{\langle S_{m}^*(\omega)\widetilde{z}(\sigma^{m}\omega), x\rangle} \\
& = \dfrac{\|T_{m}^*(\omega)\widetilde{z}(\sigma^m\omega)\|}{\|S_{m}^*(\omega)\widetilde{z}(\sigma^m\omega)\|}\dfrac{\langle \chi_m(\omega)\widetilde{z}(\omega)+\zeta_m(\omega), Z(\omega)\rangle}{\langle \widetilde{z}(\omega), Z(\omega)\rangle}\\
& = \dfrac{\|T_{m}^*(\omega)\widetilde{z}(\sigma^m\omega)\|}{\|S_{m}^*(\omega)\widetilde{z}(\sigma^m\omega)\|}\left(\chi_m(\omega) + \dfrac{\langle \zeta_m(\omega), Z(\omega) \rangle}{\langle \widetilde{z}(\omega), Z(\omega) \rangle}\right)\quad \text{ a.s.}.
\end{align*}
And this is positive for sufficiently large $m$ by the conclusion $\chi_m(\omega) = 1\ \text{ a.s.}$ and \eqref{eq:zeta.small}. The proof of \eqref{eq:psi.goal} for the case $\mathbb{E}[\xi(\omega)]<0$ follows in the same way by showing $\chi_m(\omega) = -1\ \text{ a.s.}$ when $m$ is sufficiently large.
\end{proof}

The following corollary follows from \eqref{eq:psi.goal} in the proof above.
\begin{cor}\label{cor:xy.conv}
Let $\mu$ be strong irreducible and contracting, and satisfy (FE2) and (FE4). Define $\{x_n\}_n, \{y_n\}_n$ by \eqref{defn:x} given a deterministic vector $x\not = 0$. Then
\begin{enumerate}[label=(\alph*)]
\item if $\mathbb{E}[\xi(\omega)]>0$, then $\lim_{n\to\infty} \|x_n- y_n\|_2 = 0$;
\item if $\mathbb{E}[\xi(\omega)]<0$, then $\lim_{n\to\infty} \|x_n+ y_n\|_2 = 0$.
\end{enumerate}
\end{cor}

\section{Dropping the condition $\mathbb{E}[\xi(\omega)]\not = 0$}
\label{sec:drop}

In this section, we generalise various results in \Cref{sec:varphi.x} and \Cref{sec:psi.x} by dropping the assumption $\mathbb{E}[\xi(\omega)]\not = 0$ with a simple linear trick. In particular, we will generalise \autoref{prop:ts.finite}, \autoref{thm:varphi.x}, \autoref{cor:t.asymp} \ref{t:ts} and \autoref{thm:psi}. We first remark that an informal way to see if this condition can be dropped for a certain result is by checking if $T_{n}(\omega)=0$ would make the statement ridiculously wrong, because $\mathbb{E}[\xi(\omega)]=0$ might allow $T_{n}(\omega)$ to visit $0$ infinitely often (see \Cref{sec:exmp} \autoref{exmp2}). For example, in \autoref{cor:alignment.x}, if $T_{n}(\omega)=0$, then $\delta( S_{n}(\omega)x,T_{n}(\omega)x)$ does not make sense; in \autoref{cor:t.asymp} \ref{t:range} \ref{t:tz}, $\|T_{n}(\omega)\|$ show up in the denominator; and in \autoref{cor:t.asymp} \ref{t:norm.t}, $T_{n}(\omega)=0$ gives $\log\|T_{n}(\omega)\|=-\infty$. This is basically the only obstruction to dropping the non-zero expectation condition, at least in the scope of this paper.

Let $T_{n}(\omega)=\mathfrak{D}_n(\omega, A, B)$ be the first derived product associated to $(A, B)$, and $T_{n}'(\omega)= \mathfrak{D}_n(\omega, A, B')$, where $B(\omega), B'(\omega)$ are two $d\times d$ random matrices. Then by \eqref{defn:t}
\[
T_{n}(\omega) + T_{n}'(\omega) = (B_n(\omega) + B_n'(\omega))S_{n-1}(\omega) + A_n(\omega)(T_{n-1}(\omega) + T_{n-1}'(\omega)).
\]
That is,
\begin{equation}\label{eq:add.D}
\mathfrak{D}_n(\omega, A, B) + \mathfrak{D}_n(\omega, A, B') = \mathfrak{D}_n(\omega, A, B+B').
\end{equation}
Moreover, when we take $B = \alpha A$ for some $\alpha\in \mathbb{R}$, we have
\begin{equation}\label{eq:D.alpha}
\mathfrak{D}_n(\omega, A, \alpha A) = n\alpha S_{n}(\omega).
\end{equation}
Write
\begin{align*}
\varphi_n(\omega, x, B) & := \dfrac{\| \mathfrak{D}_n(\omega, A, B)x\|_2}{\|S_{n}(\omega)x\|_2}, \\
\psi_n(\omega, x, B) & := \dfrac{\langle S_{n}(\omega)x, \mathfrak{D}_n(\omega, A, B)x\rangle}{\|S_{n}(\omega)\|_2^2},\\
\xi(\omega, B) & := \dfrac{\langle \widetilde{Z}(\sigma\omega), B(\omega)Z(\omega)\rangle}{\langle \widetilde{Z}(\sigma\omega), A(\omega)Z(\omega)\rangle}
\end{align*}
\textbf{Generalising \autoref{thm:psi}.} By \eqref{eq:add.D},
\[
\psi_n(\omega, x, B+B') = \psi_n(\omega, x, B) + \psi_n(\omega, x, B').
\]
In particular, by \eqref{eq:D.alpha}
\begin{equation}\label{eq:add.psi.n}
\psi_n(\omega, x, B + \alpha A) = \psi_n(\omega, x, B) + n\alpha.
\end{equation}
Also it is very easy to see
\[
\xi(\omega, B+B') =\xi(\omega, B) + \xi(\omega, B'), \quad \xi(\omega, \alpha A) = \alpha.
\]
Now if $\mathbb{E}[\xi(\omega, B)] = 0$, then choose $\alpha\not = 0$ such that $\mathbb{E}[\xi(\omega, B + \alpha A)] = \alpha \not = 0.$ By \autoref{thm:psi}, we have
\[
\lim_{n\to\infty}\dfrac{\psi_n(\omega, x, B+\alpha A)}{n} = \mathbb{E}[\xi(\omega, B + \alpha A)] = \alpha\quad \text{ a.s.}.
\]
Using \eqref{eq:add.psi.n} to rewrite the left-hand side, we have
\[
\lim_{n\to\infty}\dfrac{\psi_n(\omega, x, B)}{n} = 0 = \mathbb{E}[\xi(\omega)]\quad \text{ a.s.}.
\]
This gives the generalisation of \autoref{thm:psi}.
\\
\textbf{Generalising \autoref{thm:varphi.x}.} 
If $\mathbb{E}[\xi(\omega, B)] = 0$, then choose $\alpha > 0$ such that 
\[
|\mathbb{E}[\xi(\omega, B \pm \alpha A)]| = \alpha\not = 0.
\]
Moreover, by \eqref{eq:add.D}
\[
2\varphi_n(\omega, x, B) \le \varphi_n(\omega, x, B+\alpha A) + \varphi_n(\omega, x, B - \alpha A).
\]
Dividing both sides by $n$ and letting $n\to\infty$ using \autoref{thm:varphi.x} gives
\[
\limsup_{n\to\infty}\dfrac{\varphi_n(\omega, x, B)}{n}\le \alpha\quad \text{ a.s.}.
\]
As $\alpha>0$ is arbitrary, we have
\[
\lim_{n\to\infty}\dfrac{\varphi_n(\omega, x, B)}{n} = 0 = |\mathbb{E}[\xi(\omega)]|.
\]
\autoref{cor:t.asymp} \ref{t:ts} could be generalised in the same way. And the generalised \autoref{prop:ts.finite} follows from the generalised \autoref{cor:t.asymp} \ref{t:ts}.

\section{Examples}
\label{sec:exmp}
In this section, we give several relevant examples. \autoref{exmp1} is a trivial example of commutative scalar products. \autoref{exmp2} gives an example of $T_{n}(\omega)$ visiting $0$ infinitely often almost surely. \autoref{exmp:no.irred} and \autoref{exmp:no.contra} demonstrate that our theory does not work if discarding strong irreducibility or contracting property.

\begin{exmp}\label{exmp1}
The most trivial example is when $d=1$ and random matrices reduce to random scalars. Assume $\{(a_i(\omega), b_i(\omega))\}_{i\ge 1}$ is an i.i.d. sequence of scalar pairs following a probability measure $\mu$ and $a_i(\omega)\not = 0$. Now
\begin{align*}
S_{n}(\omega) & = a_n(\omega)\cdots a_1(\omega),\\
T_{n}(\omega) &= \left(\sum_{i=1}^n b_i(\omega)/a_i(\omega)\right) S_n(\omega).
\end{align*}
In this case,
\begin{align*}
\dfrac{\varphi_n}{n}&=\dfrac{|T_{n}(\omega)|}{n|S_{n}(\omega)|} =\dfrac{1}{n}\left|\sum_{i=1}^n\dfrac{b_i(\omega)}{a_i(\omega)}\right|\to \left|\mathbb{E}\left[\dfrac{b_1(\omega)}{a_1(\omega)}\right]\right|\ (n\to\infty) \  \text{ a.s.} \\
\dfrac{\psi_n}{n}&=\dfrac{S_{n}(\omega)T_{n}(\omega)}{n|S_{n}(\omega)|^2} =\dfrac{1}{n}\sum_{i=1}^n\dfrac{b_i(\omega)}{a_i(\omega)} \to \mathbb{E}\left[\dfrac{b_1(\omega)}{a_1(\omega)}\right]\ (n\to\infty) \  \text{ a.s.} 
\end{align*}
\end{exmp}

\begin{exmp}\label{exmp2}
Let $\mu$ be the probability measure on pairs of matrices under which there is a probability of $1/2$ to choose $(P, P)$ and a probability of $1/2$ to choose $(Q, -Q)$. Assume the matrices $\{P, Q\}$ are invertible and form a strongly irreducible and contracting set.

In this example, for a product $S_{n}(\omega)$ consisting of $\eta_P(\omega)$ number of $P$'s and $\eta_Q(\omega)$ of $Q$'s, then $T_{n}(\omega) = (\eta_P(\omega) - \eta_Q(\omega))S_{n}(\omega)$, where $\eta_P(\omega)+\eta_Q(\omega) = n$. Now for a given vector $x \not = 0$,
\[
\psi_n(\omega, x)=\dfrac{\langle S_{n}(\omega)x, T_{n}(\omega)x\rangle}{\|S_{n}(\omega)x\|_2^2} = \eta_P(\omega) - \eta_Q(\omega).
\]
That is, $\psi(\omega, x)$ is a simple random walk on $\mathbb{Z}$. Therefore we know by the basic theory of random walks that $\mathbb{P}(T_{n}(\omega)=0 \ \text{i.o.})=\mathbb{P}(\psi_n(\omega, x)=0\ \text{i.o.}) = 1$. Also $\psi_n(\omega, x)/n\to 0$ almost surely as $n\to\infty$. Meanwhile,
\[
\mathbb{E}[\xi(\omega)] = \dfrac{1}{2}\int\dfrac{\langle \widetilde{Z}(\sigma\omega), PZ(\omega) \rangle}{\langle \widetilde{Z}(\sigma\omega), PZ(\omega)\rangle}\mathbb{P}(\operatorname{d}\omega) + \dfrac{1}{2}\int\dfrac{\langle \widetilde{Z}(\sigma\omega), -QZ(\omega) \rangle}{\langle \widetilde{Z}(\sigma\omega), QZ(\omega)\rangle}\mathbb{P}(\operatorname{d}\omega) = 0.
\]
\end{exmp}

\begin{exmp}\label{exmp:no.irred}
Here is an example with contracting property but without strong irreducibility. Let $A_n(\omega)$ and $B_n(\omega)$ be deterministic matrices given by 
\[
A_n(\omega)=\begin{pmatrix}
\alpha & 0\\
0 & \beta
\end{pmatrix},\quad 
B_n(\omega)= R_{\pi/2} = \begin{pmatrix}
0 & -1 \\
1 & 0
\end{pmatrix}
\]
where $\alpha>\beta>0$. 
Then $S_{n+1}(\omega)= \operatorname{diag}(\alpha^{n+1}, \beta^{n+1})$, while
\begin{align*}
T_{n+1}(\omega) = \begin{pmatrix}
0 & -\sum_{i=0}^n \alpha^{n-i}\beta^i \\
\sum_{i=0}^n \beta^{n-i}\alpha^i & 0
\end{pmatrix}.
\end{align*}
Then for the standard basis $e_1, e_2$ of $\mathbb{R}^2$, we have
\begin{align*}
\dfrac{1}{n}\dfrac{\|T_{n+1}(\omega)e_1\|_2}{\|S_{n+1}(\omega)e_1\|_2} &= \dfrac{1}{n\alpha}\sum_{i=0}^n \left(\dfrac{\beta}{\alpha}\right)^i \to 0, \\
\dfrac{1}{n}\dfrac{\|T_{n+1}(\omega)e_2\|_2}{\|S_{n+1}(\omega)e_2\|_2} &= \dfrac{1}{n\beta}\sum_{i=0}^n \left(\dfrac{\alpha}{\beta}\right)^i \to \infty,
\end{align*}
as $n\to\infty$.
\end{exmp}

\begin{exmp}\label{exmp:no.contra}
This is an example with strong irreducibility but without contracting property. Let $A_n(\omega)$ and $B_n(\omega)$ be deterministic matrices given by
\[
A_n(\omega) = R_\theta = \begin{pmatrix}
\cos\theta & -\sin\theta \\
\sin\theta & \cos\theta
\end{pmatrix}, \quad
B_n(\omega) = I_2= \begin{pmatrix}
1 & 0 \\
0 & 1
\end{pmatrix},
\]
where $\theta/\pi$ is irrational. Denote by $x_\alpha : = (\cos\alpha) e_1 + (\sin\alpha) e_2$ the normalised vector with angle $\alpha$, where $e_1, e_2$ are standard basis for $\mathbb{R}^2$. Fix an $\alpha$, then it is easy to observe that since $S_{n}(\omega)=R_{n\theta}$ and $T_{n}(\omega) = nR_{(n-1)\theta}$, $S_{n}(\omega)x_\alpha$ and $T_{n}(\omega)x_\alpha$ always differ by an angle of $\theta$ in direction. Thus $S_{n}(\omega)x_\alpha$ and $T_{n}(\omega)x_\alpha$ do not align as $n\to\infty$ in this case.

Moreover, for all $n\ge 1$, since a rotation matrix is orthogonal,
\[
\dfrac{\varphi_n(\omega, x_\alpha)}{n} = \dfrac{1}{n}\dfrac{\|T_{n}(\omega)x_\alpha\|_2}{\|S_{n}(\omega)x_\alpha\|_2} = 1.
\]
And since
\[
\langle S_{n}(\omega)x_\alpha, S_{n-i}(\sigma^{i}\omega)I_2S_{i-1}(\omega)x_\alpha\rangle = \langle x_{\alpha+i\theta}, x_{\alpha+(i-1)\theta} \rangle = \cos\theta.
\]
This shows
\[
\dfrac{\psi_{n}(\omega, x_\alpha)}{n} = \dfrac{1}{n}\sum_{i=1}^n \cos\theta = \cos\theta.
\]
As $\theta$ is an arbitrary irrational multiple of $\pi$, it is impossible to have $1 = |\cos \theta|$. This is not surprising because $\varphi_n$ ignores the angle difference between $S_{n}(\omega)x_\alpha$ and $T_{n}(\omega)x_\alpha$ while $\psi_n$ preserves it.
\end{exmp}

\bibliographystyle{plain}
\bibliography{random_matrices,operator_algebras}
\end{document}